\documentclass[10pt]{amsart}
\usepackage{Angelini-Knoll}
 
\author{Gabe Angelini-Knoll}
\title{On Topological Hochschild homology of the $K(1)$-local sphere}

\def\Sp{\operatorname{Sp}}  

\usepackage[colorlinks=true, linkcolor=blue, menucolor=blue]{hyperref}

\usepackage{cleveref}
    
\usepackage{amsrefs}
\begin{document}
\singlespacing
\maketitle
\begin{abstract}
We compute mod $(p,v_1)$ topological Hochschild homology of the connective cover of the $K(1)$-local sphere spectrum for all primes $p\ge 3$. This is accomplished using a May-type spectral sequence in topological Hochschild homology constructed from a filtration of a commutative ring spectrum.  
\end{abstract}

\tableofcontents

\section{Introduction}
Algebraic K-theory of rings is known to encode deep arithmetic information; for example, algebraic K-theory groups of rings of integers in totally real number fields are related to special values of Dedekind zeta functions by the Lichtenbaum--Quillen conjecture. Following the Ausoni--Rognes program \cite{MR1947457}, we would like to explore the arithmetic encoded in the algebraic K-theory of ``brave new rings'' or, more precisely, 
ring spectra. To approach this, we use a technique, initiated by B\"okstedt in \cite{bok2}, where one approximates algebraic K-theory by topological Hochschild homology. Specifically, topological Hochschild homology is a linear approximation to algebraic K-theory in the sense of Goodwillie's calculus of functors by \cite{MR1307900}.
The purpose of this paper is to compute topological Hochschild homology of the connective cover of the $K(1)$-local sphere mod $(p,v_1)$. Here $K(1)$ is the first Morava K-theory which has coefficients $K(1)_*\cong \mathbb{F}_p[v_1^{\pm 1}]$.  

This computation uses a spectral sequence associated to a multiplicative filtration constructed by the author and A. Salch in \cite{thhmay}. The idea of the spectral sequence is to mimic the construction of May where he filters a Hopf algebra by powers of the augmentation ideal and constructs an associated filtration of the bar construction. In particular, the author and A. Salch show in \cite{thhmay} that there is a model for the Whitehead tower of a connective commutative ring spectrum as a filtered commutative ring spectrum. From this filtered commutative ring spectrum, the author and A. Salch produce a filtration of a generalized bar construction, for example the cyclic bar construction, and construct a spectral sequence in (higher) topological Hochschild homology. This work is summarized in Section \ref{section 2}. 

In the 1980's, F. Waldhausen first suggested computing algebraic K-theory of connective and non-connective versions of the localizations of the sphere spectrum appearing in the chromatic tower as an approach to computing algebraic K-theory of the sphere spectrum \cite{MR764579}. There are homology theories $E(n)$, depending on a nonnegative integer $n$ and a prime $p$, called the Johnson-Wilson $E$-theories with coefficients $\mathbb{Z}_{(p)}[v_1,v_2,\ldots , v_{n-1},v_n^{\pm 1}]$. Using Bousfield localization one can construct spectra $L_{E(n)}S_{(p)}$ with maps $L_{E(n)}S_{(p)}\rightarrow L_{E(n-1)}S_{(p)}$ for each $n$. The homotopy limit gives a good approximation to the sphere spectrum in the sense that 
\[ S_{(p)} \lra \underset{\leftarrow}{\holim}\hspace{.05in} L_{E(n)}S_{(p)} \]
is a weak equivalence due to the chromatic convergence theorem of M. Hopkins and D. Ravenel \cite[Thm. 7.5.7]{rav2}. In the connective case, the limit is also known to converge after applying algebraic K-theory in the sense that the map 
\[ K(S_{(p)})\lra \underset{\leftarrow}{\holim}\hspace{.05in} K(\tau_{\ge 0} L_{E(n)}S_{(p)} ) \]
is a weak equivalence \cite[Thm. 2.8]{MR1164148}, where $\tau_{\ge 0}L_{E(n)}S_{(p)}$ is the connective cover of $L_{E(n)}S_{(p)}$. By convention $\tau_{\ge 0}L_{E(0)}S_{(p)}$ is $H\mathbb{Z}_{(p)}$ and the first map is the linearization map $\tau_{\ge 0}L_{E(1)}S_{(p)}\rightarrow H\mathbb{Z}_{(p)}$. This gives a method for organizing the information in the algebraic K-theory of the sphere spectrum in a similar way to how the $E(n)$-localizations organize the information in the homotopy groups of spheres according to ``chromatic height.''

We can 
approximate $K(\tau_{\ge 0}L_{E(1)}S_{(p)})$ using the trace methods approach initiated by B\"okstedt  \cite{bok2}. In other words, we can 
approximate algebraic K-theory by topological cyclic homology, which is an invariant that is built out of topological Hochschild homology (THH) using the cyclotomic structure on THH. Since topological cyclic homology isn't sensitive to the difference between $R$ and its $p$-completion $R_p$, 
we will compute algebraic K-theory of 
the $p$-completion $(\tau_{\ge 0}L_{E(1)}S_{(p)})_p$, 
which is equivalent to $\tau_{\ge 0}L_{K(1)}S$. Moreover, one can compute algebraic K-theory of $\tau_{\ge 0}L_{E(1)}S_{(p)}$ using the homotopy pullback 
\[ 
	\xymatrix{ 
		K(\tau_{\ge 0}L_{E(1)}S_{(p)})_p \ar[r] \ar[d] & K(\tau_{\ge 0}L_{K(1)}S)_p \ar[d] \\
		K(H\mathbb{Z}_{(p)})_p  \ar[r] & K(H\mathbb{Z}_p)_p
		}
\]
due to Dundas, which appears in \cite[Thm. 6.7]{MR2079370}.

We now briefly outline the trace methods approach for computing algebraic K-theory of $p$-complete commutative ring spectra. We may model topological Hochschild homology of a commutative ring spectrum $R$ as the tensoring $S^1\otimes R$, in commutative ring spectra, with the circle
by the main theorem of \cite{MR1473888}. The group $S^1$ therefore acts on $S^1\otimes R$ on the first coordinate. Additionally, $THH(R)$ is a cyclotomic spectrum 
which allows us to produce maps 
\[ \xymatrix{ THH(R)^{C_{p^n}} \ar@<.5ex>[r]^(.45){F} \ar@<-.5ex>[r]_(.45){R} & THH(R)^{C_{p^{n-1}}} }  \]
where $F$ is given by inclusion of fixed points and the map $R$ is constructed using the cyclotomic structure. We may then form $p$-typical topological cyclic homology
\[ TC(R;p) := \underset{\underset{F,R}{\longleftarrow}}{\lim}\thinspace THH(R)^{C_{p^n}}.\]
For more details about this construction, see \cite{MR1410465}. 

The key step in approximating algebraic K-theory is the use of the theorem of B. Dundas- T. Goodwillie, and R. McCarthy \cite{DGM}, which together with work of L. Hesselholt and I. Madsen \cite{MR1410465} produces the following result
\begin{thm}[Theorem VII 3.1.14 \cite{DGM} ] \label{DGM thm}
If  $R$ is a connective commutative ring spectrum and there is a surjection $\mathbb{Z}_p\to \pi_0R$ then there is an equivalence 
\[ K(R)_p\simeq \tau_{\ge 0} TC(R;p)_p \] 
where $\tau_{\ge 0}$ is the connective cover functor.
\end{thm}

The spectrum $\tau_{\ge 0}L_{K(1)}S$ is also modeled by algebraic K-theory of certain finite fields. For this explanation, let $p$ be an odd prime and let $q$ be a prime power that also topologically generates $\mathbb{Z}_p^{\times}$. Due to D. Quillen \cite{MR0315016}, it is known that algebraic K-theory of a finite field of order $q$  can be computed, after $p$-completion, using the fiber sequence 
\[ \xymatrix{ K(\mathbb{F}_q)_p \ar[r] & ku_p \ar[r]^{1-\psi_q} & \Sigma^{2}ku_p} \] 
$\psi_q$ is the $q$-th Adams operation and $ku_p$ is the $p$-completion of connective complex K-theory. Now, E. Devinatz and M. Hopkins showed in \cite{MR2030586} that there is an equivalence 
\[ L_{K(1)}S \simeq KU_p^{h\mathbb{G}_1} \] 
where $\mathbb{G}_1$ is the height $1$ Morava stabilizer group, which is isomorphic to $\mathbb{Z}_p^{\times}$, and $KU_p$ is the $p$ completion of periodic complex $K$-theory. Note that the Morava stabilizer group $\mathbb{G}_1=\mathbb{Z}_p^{\times}$ acts on $KU_p$ by Adams operations, so we may write $\psi_{q}$ for the Adams operation corresponding to the topological generator $q$ of $\mathbb{Z}_p^{\times}$ \cite{MR3328537}. The homotopy fixed points can therefore be modeled by the fiber sequence
\[ \xymatrix{ L_{K(1)}S \ar[r] & KU_p \ar[r]^{1- \psi_{q} } & KU_p. }\] 
Under the stated conditions on $q$ and $p$, there is therefore a map of fiber sequences  
\[ 
\xymatrix{
K(\mathbb{F}_q)_p \ar[r] \ar[d] & ku_p \ar[r]^{1-\psi_q} \ar[d] & \Sigma^2 ku_p \ar[d]\\
L_{K(1)}S \ar[r] & KU_p \ar[r]^{1-\psi_q} &  KU_p . } \] 
By examining homotopy groups of each of the fibers, we see that the map of fibers induces an isomorphism in homotopy in degrees greater or equal to zero. The homotopy groups of $L_{K(1)}S$ agree with Quillen's computation 
\[ K_*(\mathbb{F}_q) \cong \left \{ \begin{array}{ll} \mathbb{Z} & \text{ if } *=0  \\ \mathbb{Z}/(q^i-1)\mathbb{Z} & \text{ if } *=2i-1  \text{ for  } i \in \mathbb{N} \end{array} \right. \] 
in nonnegative degrees after $p$-completion when $q$ is a prime power that topologically generates $\mathbb{Z}_p^{\times}$. Therefore, there are equivalences 
$K(\mathbb{F}_q)_p \simeq \tau_{\ge 0}L_{K(1)}S\simeq (\tau_{\ge 0}L_{E(1)}S)_p$, where $\tau_{\ge 0}L_{K(1)}S$ denotes the connective cover of the $K(1)$-local sphere. See the proof of Theorem 5 and the Introduction of \cite{MR2030586} for more details. The upshot of this discussion is that the main theorem of this paper a first approximation to iterated algebraic K-theory of a large class of finite fields.

We also compute mod $(p,v_1)$ homotopy of $THH(\tau_{\ge 0}L_{K(1)}S)$ because this is the natural first step in approaching the chromatic red-shift conjecture of C. Ausoni and J. Rognes \cite{MR2499538}. The homology theories  $K(n)$ for each prime $p$ and nonnegative integer $n$ have coefficients
$K(n)_*\cong  \mathbb{F}_p[v_n^{\pm 1}]$ for $n\ge 1$ and   $K(0)_*\cong \mathbb{Q}$ and they are useful for detecting chromatic height. 
We say that a $p$-local finite cell $S$-module $V$ has type $n$ if $K(n-1)_*V\cong 0$ and $K(n)_*V$ is nontrivial. Due to the thick subcategory theorem of M. Hopkins and J. Smith \cite[Thm. 7]{MR1652975}, the category of $p$-local finite cell $S$-modules $\mathcal{F}_{(p)}$ can be filtered into proper thick subcategories 
\[ 0 \subset \dots \subset \cC_2 \subset \cC_1 \subset \cC_0 = \mathcal{F}_{(p)} \] 
where $\cC_n$ contains exactly spectra of type $\ge n$; i.e., it consists of the $K(n-1)$-acyclic spectra \cite{MR960945}. The first examples of type $n$ spectra for small $n$ can be constructed by an iterative procedure by taking cofibers of $v_n$-self maps beginning with $v_0=p$: for example, the cofiber of the multiplication by $p$ map, denoted $S/p$, is a type one spectrum. At odd primes, we can continue this process and construct a type one spectrum $V(1)$ as the cofiber of a periodic self map $v_1\co \Sigma^{2p-2}S/p\rightarrow S/p$.  By M. Hopkins and J. Smith \cite[Thm. 4.11]{MR1652975}, there exists a type $n$ spectrum for each $n$, though not necessarily constructed by the iterative procedure we just discussed. 


To phrase the red-shift conjecture, it is necessary to have a notion of chromatic height for spectra that are not finite cell $S$-modules. We recall a definition due to N. Baas, B. Dundas, and J. Rognes that provides a notion of chromatic height in this context. This uses the fact that every $p$-local finite cell $S$-module $V$ admits a $v_n$-self map for some $n$, a consequence of the periodicity theorem of M. Hopkins, and J. Smith \cite[Thm. 9]{MR1652975}. 
\begin{defin}[Baas--Dundas--Rognes {\cite[Def. 6.1]{MR2079370}}]
Let $X$ be a spectrum and let $\mathcal{T}_X$ be the thick subcategory of finite $p$-local spectra $V$ such that the map
\begin{equation} \label{eq red shift} V\wedge X \rightarrow v_n^{- 1} V\wedge  X \end{equation}
induces an isomorphism in homotopy groups in sufficiently high degrees. Then if $\mathcal{T}_X=\cC_{n}$, we say that $X$ has telescopic complexity $n$. 
\end{defin} 
We now state the conjecture due to C. Ausoni and J. Rognes, which generalizes the Lichtenbaum--Quillen conjecture to higher chromatic heights. 
\begin{conj}[Ausoni--Rognes red-shift conjecture]
If $R$ is a suitably finite $K(n)$-local spectrum, then $K(R)$ has telescopic complexity $n+1$. 
\end{conj}
Note that this version of the conjecture does not apply to the connective cover of the $K(1)$-local sphere directly, however, we can relate it to algebraic K-theory of the $K(1)$-local sphere using a localization sequence in algebraic K-theory \'a la Waldhausen \cite{MR802796}. Using a fracture square argument, we observe that $L_{E(1)}(\tau_{\ge 0}L_{K(1)}S)\simeq L_{K(1)} (\tau_{\ge 0}L_{K(1)}S)\simeq L_{K(1)}S$ and there is therefore a fiber sequence 
\[ K(\mathcal{D}_1)\to K(\tau_{\ge 0}L_{K(1)}S ) \to  K(L_{K(1)}S ) \]
where $\mathcal{D}_1$ is the category of finite $\tau_{\ge 0}L_{K(1)}S$-modules that are $E(1)$-acyclic. This relies on the fact that the telescope conjecture is true at height one and consequently $E(1)$-localization is a finite localization. Using a d\'evissage argument as in Blumberg-Mandell \cite{MR2413133}, one could identify $K(\mathcal{D}_1)$ as algebraic K-theory of a spectrum and then use the localization sequence to compute mod $(p,v_1)$ homotopy of $K(L_{K(1)}S)$. Thus, by determining the telescopic complexity of $K(\tau_{\ge 0} L_{K(1)}S)$ we can verify the Ausoni--Rognes red-shift conjecture in this case. 

We may also phrase the red-shift conjecture using telescopic complexity, as C. Barwick does in \cite{Bar14} in the case of iterated algebraic K-theory of the complex numbers.
\begin{conj}[Telescopic red-shift conjecture] \label{conj}
If $R$ is a commutative ring spectrum with telescopic complexity $n$, then $K(R)$ is a commutative ring spectrum with telescopic complexity $n+1$.
\end{conj} 
In the example of interest, we know that 
\[ \pi_k  \left ( S/p \wedge \tau_{\ge 0}L_{K(1)}S \right ) \longrightarrow \pi_k \left ( v_1^{-1} S/p \wedge \tau_{\ge 0}L_{K(1)}S \right ) \]
is an  isomorphism for $k$ sufficiently large. Therefore, $\tau_{\ge 0}L_{K(1)}S$ has telescopic complexity one. Determining the telescopic complexity of $K(\tau_{\ge 0}L_{K(1)}S)$ 
allows us to verify the telescopic complexity red-shift conjecture directly in this case. To determine the telescopic complexity of $K(\tau_{\ge 0}L_{K(1)}S)$, we need to know that 
\begin{equation} \label{eq case of red shift} \pi_k \left ( V(1)\wedge K(\tau_{\ge 0}L_{K(1)}S) \right )\rightarrow \pi_k \left ( v_2^{-1}V(1)\wedge K(\tau_{\ge 0}L_{K(1)}S) \right )\end{equation}
is an isomorphism for $k$ sufficiently large.  
The author plans to verify this in future work as a natural extension of the results of this paper. 

In the present paper, we compute topological Hochschild homology of the connective $K(1)$-local sphere, after smashing with the Smith-Toda complex $V(1)$, as follows. 
\begin{thm}\label{main thm} Let $p$ be an odd prime. There is an isomorphism of graded rings
\[ V(1)_*(THH(\tau_{\ge 0}L_{K(1)}S))\cong P(\mu_2)\otimes \Gamma(\sigma b) \otimes \mathbb{F}_p\{\alpha_1, \lambda_1', \lambda_2\alpha_1, \lambda_2\lambda_1', \lambda_2\lambda_1'\alpha_1 \}  \]
where the products between the classes 
$ \{\alpha_1, \lambda_1', \lambda_2\alpha_1, \lambda_2\lambda_1', \lambda_2\lambda_1'\alpha_1 \} $
are zero except for 
\[ \alpha_1\cdot \lambda_2\lambda_1'=\lambda_1'\cdot \lambda_2\alpha_1 =\lambda_2\lambda_1'\alpha_1.  \] 
\end{thm} 
The paper is divided into two sections. Section \ref{section 2} gives a brief summary of the construction of the topological Hochschild-May  spectral sequence and provides a large class of examples of decreasingly filtered commutative ring spectra; i.e., those that can be produced as a multiplicative model for the Whitehead tower of a connective commutative ring spectrum. In Section \ref{section 3}, we provide all the details needed to prove Theorem \ref{main thm}. 
\subsection{Notation and Conventions} 
Throughout, let $\Sp$ be the category of symmetric spectra of pointed simplicial sets. We equip $\Sp$ with the positive stable flat model structure \cite{schwedebook}. Note that $\Sp$ is a closed symmetric monoidal model category with $\wedge$ as symmetric monoidal product and the sphere spectrum $S$ as the unit. We will write $\Comm \mathcal{D}$, for the category of commutative monoids in a symmetric monoidal category $\mathcal{D}$. When a map $X\rightarrow Y$ between objects in $\Sp$ is a cofibration, we will write $Y/X$ for the cofiber.  We will fix a prime $p\ge 3$ throughout. This ensures that the mod $p$ Moore spectrum has a $v_1$-self map $v_1\co \Sigma^{2p-2}S/p\lra S/p$ and the cofiber, denoted $V(1)$, exists. 

As we observed in the introduction, there is an equivalence $K(\mathbb{F}_q)_p\simeq \tau_{\ge 0} L_{K(1)}S$. There is also an equivalence $\tau_{\ge 0}L_{K(1)}S\simeq j_p$ where $j_p$ is the $p$ completion of the connective image of $J$ spectrum. We use the model $K(\mathbb{F}_q)_p$, which is known to be a commutative ring spectrum and we write $j$ throughout for a cofibrant replacement of this ring spectrum in $\Comm \Sp$. 

Throughout, we will write $\dot{=}$ to mean equivalence up to multiplication by a unit in $\mathbb{F}_p$.  We write $E(x_1,x_2, ...)$ for an exterior algebra over $\mathbb{F}_p$, $P(x_1,x_2,...)$ for a polynomial algebra over $\mathbb{F}_p$, $P_p(x_1,x_2, ...)$ for a truncated polynomial algebra over $\mathbb{F}_p$ truncated at the $p$-th power, and $\Gamma(x_1,x_2,...)$ for a divided power algebra over $\mathbb{F}_p$ on generators $x_i$ for $i\ge 1$. We recall that a divided power algebra generated by $x$ has generators $\gamma_i(x)$ for $i\ge 0$ with relations $\gamma_0(x)=1$, $\gamma_1(x)=x$ and $\gamma_i(x)\gamma_j(x)=(i,j)\gamma_{i+j}(x)$ where 
$(i,j)=\binom{i+j} {j}.$
In particular, over the finite field $\mathbb{F}_p$, there is an isomorphism
$ \Gamma(z)\cong P_{p}(z, \gamma_p(z) ,\gamma_{p^2}(z) , ...). $
When not otherwise indicated, $\otimes$ will represent $\otimes_{\mathbb{F}_p}$ and $HH_*(R)$ will represent $HH_*^{\mathbb{F}_p}(R)$ when $R$ is a (bi)graded $\mathbb{F}_p$-algebra. 
\subsection{Acknowledgements}
It is a pleasure to thank Andrew Salch for many conversations about material related to this paper. The author also benefited from discussions with  Robert Bruner, Bjorn Dundas, Eva H\"oning, and Dan Isaksen. The author would also like to thank an anonymous referee for a careful reading of the paper leading to several improvements.

\section{Topological Hochschild-May spectral sequence} \label{section 2}
The goal of this section is to briefly summarize the necessary ingredients for the topological Hochschild-May  spectral sequence. This section is a review of results from the author's paper with A. Salch \cite{thhmay} and we refer the reader to that paper for details. To describe the spectral sequence, we first need to define what we mean by a filtered object in $\Sp$. We write $(\mathbb{N}^{\op},+,0)$ for the opposite category of the natural numbers $\mathbb{N}$, regarded as a partially ordered set, with symmetric monoidal product given by summing natural numbers. 
\begin{defin}  \label{def:DFO} A \it{decreasingly filtered commutative monoid} \rm in $\Sp$ is a lax symmetric monoidal functor 
\[ I \co (\mathbb{N}^{\op},+,0) \lra (\Sp,\wedge ,S). \] 
\end{defin}  
We will write $I_n$ for a decreasingly filtered commutative monoid in $\Sp$ evaluated on a natural number  $n\in\mathbb{N}^{\op}$. Concretely, a decreasingly filtered commutative monoid in $\Sp$ is a sequence of objects in $\Sp$
\[ \xymatrix{ \ldots \ar[r]^{f_3} & I_2 \ar[r]^{f_2} & I_1 \ar[r]^{f_1} & I_0 } ,\]
along with structure maps 
\[ \rho_{i,j}\co I_i \wedge I_j \lra I_{i+j} \] 
satisfying certain commutativity, associativity, unitality, and compatibility axioms, which are encoded in the structure of a lax symmetric monoidal functor. 
By B. Day \cite[Ex. 3.2.2]{Day}, the category of decreasingly filtered commutative monoids in $\Sp$ is equivalent to $\Comm \Sp^{\mathbb{N}^{\op}}$. In Section 4.1 of \cite{thhmay}, a useful model structure on $\Comm \Sp^{\mathbb{N}^{\op}}$ is discussed. The cofibrant objects in the model structure described in Section 4.1 of \cite{thhmay} will be referred to as {\em cofibrant decreasingly filtered commutative monoids in $\Sp$} and the reader can see \cite{thhmay} for more details since the model structure will not play an important role in the present article.
Naturally, one would like to define the associated graded commutative monoid in $\Sp$ of a cofibrant decreasingly filtered commutative monoid in $\Sp$. The maps $\rho_{i,j}$ are the structure necessary to make sense of such an associated graded commutative monoid in $\Sp$. Additively, we define the \emph{associated graded commutative monoid of $I_{\bullet}$} to be 
\[ E_0^*I := \coprod_{n\in\mathbb{N}} I_n/I_{n+1},\] 
and in \cite[Def. 3.1.6]{thhmay} we provide $E_0^*I$ with the structure of a commutative monoid in $\Sp$, such that the multiplication maps are induced by maps
\[ I_s/I_{s+1}\wedge I_t/I_{t+1} \to I_{s+t}/I_{s+t+1}. \] 

\begin{defin}
Let $R$ be an object in $\Comm \Sp$ and let $X_{\bullet}$ be a simplicial finite set. We define the tensor product 
$ X_{\bullet} \otimes R $ 
to be the realization of the simplicial object 
$X_{\bullet} \tilde{\otimes} R$ in $\Sp$
where the $n$-simplices are 
$ (X_{\bullet} \tilde{\otimes} R)_n=\bigwedge_{s\in X_n} R\{s\}$
with face maps 
$ d_i\co (X_{\bullet} \tilde{\otimes} R)_n\lra (X_{\bullet} \tilde{\otimes} R)_{n-1}$
given on each summand of the coproduct by the map
$ R\{s\}\lra \bigwedge_{t\in X_{n-1}} R\{t\} $
which includes $R\{s\}$ into the summand corresponding to $\delta_i(s)\in X_{n-1}$. Here the maps $\delta_i$ and $\sigma_i$ are the face and degeneracy maps of the simplicial finite set $X_{\bullet}$. Note that the coproduct in commutative ring spectra is the smash product and we are using that fact here. Similarly, the degeneracy map 
$s_i\co (X_{\bullet} \tilde{\otimes} R)_{n-1}\lra (X_{\bullet} \tilde{\otimes} R)_{n}$
is given on each summand by the map 
$ R\{t\} \lra \bigwedge_{s\in X_n} R\{s\}$
where $R\{t\}$ includes as the smash factor corresponding to $\sigma_i(t)\in X_n$. 
\end{defin}
\begin{exm} 
In the case where $X_{\bullet}=\Delta[1]/\delta \Delta[1]=:S^1_{\bullet}$, the minimal simplicial model for the circle,  
\[ THH(R) : = S^1_{\bullet} \otimes R \] 
is the geometric realization of the simplicial object in $\Sp$
\begin{equation}\label{}
S^1_{\bullet} \tilde{\otimes} R := \left \{  \xymatrix{ R  & \ar@<1ex>[l]\ar@<-1ex>[l] 
 R\wedge R &                         
 \ar@<2ex>[l]\ar@<-2ex>[l]  \ar[l]R\wedge R \wedge R& 
    \ar@<3ex>[l]\ar@<-3ex>[l]\ar@<1ex>[l]\ar@<-1ex>[l]  \dots   }  \right \} \end{equation}
    with face and degeneracy maps given by the following formulas: the face maps are  
 \[ 
 d_i=\left \{ \begin{array}{ll} 
 \id_{R} \wedge  \ldots \id_{R}  \wedge \mu \wedge \id_{R} \wedge \ldots \wedge \id_{R}  & \text{ if } i <n \\ 
(\mu \wedge   \id_{R}   \wedge  \ldots   \wedge \id_{R}) \circ t_{n}  &  \text{ if } i=n\end{array} \right .  \] 
where the multiplication map $\mu\co R\wedge R\rightarrow R$ is in the $i$-th position on the first line and 
\[t_n\co R^{\wedge n} \rightarrow R^{\wedge n}\] 
is the map that cyclicly permutes the factors to the right.  The degeneracy maps are 
\[ s_i= \id_R \wedge \ldots \wedge \id_R \wedge \eta \wedge \id_R \wedge \dots \wedge \id_R \] 
where the unit map $\eta\co S\rightarrow R$ from the sphere spectrum is in the $i$-th position. 
\end{exm}
\begin{thm}[Theorem 4.2.1 \cite{thhmay}]
Let $I$ be a cofibrant decreasingly filtered commutative monoid in $\Sp$ and let $X_{\bullet}$ be a simplicial finite set then there is a spectral sequence
\begin{equation} \label{thh may ss} E^1_{s,t} = E_{s,t}(X_{\bullet}\otimes E_0^*(I)))\Rightarrow E_s (X_{\bullet}\otimes I_0) \end{equation} 
with differential 
\[ d^r\co E^r_{s,t}  \lra E^r_{s-1,t+r}  \] 
for any connective generalized homology theory $E_*$, where the second grading on the input of the spectral sequence keeps track of the May filtration. 
The spectral sequence strongly converges when 
$\pi_k(I_i)\cong 0$ for $k<i$  and the differentials satisfy the Leibniz rule. The spectral sequence is also functorial in both variables in the sense that a map of simplicial finite sets $X_{\bullet}\to Y_{\bullet}$ induces a map of spectral sequences 
\[ \xymatrix{
E_{s,t}(X_{\bullet}\otimes E_0^*(I)))\ar[d] \ar@{=>}[r] & E_s (X_{\bullet}\otimes I_0) \ar[d] \\
E_{s,t}(Y_{\bullet}\otimes E_0^*(I))) \ar@{=>}[r] & E_s (Y_{\bullet}\otimes I_0)  
}
\]
and a map $I\to J$ of cofibrant decreasingly filtered commutative monoids in $\Sp$ induces a map of spectral sequences 
\[ \xymatrix{
E_{s,t}(X_{\bullet}\otimes E_0^*(I)))\ar[d] \ar@{=>}[r] & E_s (X_{\bullet}\otimes I_0) \ar[d] \\
E_{s,t}(X_{\bullet}\otimes E_0^*(J))) \ar@{=>}[r] & E_s (X_{\bullet}\otimes J_0). 
}
\]
\end{thm}
In this paper, we will primarily be concerned with the case where $X_{\bullet}=S^1_{\bullet}$, in which case the spectral sequence \eqref{thh may ss} reduces to 
\begin{equation}
 E^1_{s,t} = E_{s,t}(THH(E_0^*(I)))\Rightarrow E_s (THH(I_0)).
 \end{equation}
We retain the full generality of the theorem here though because it will also be useful to consider the case where $X_{\bullet}=*$. We will also use the map of topological Hochschild-May  spectral sequences induced by $*\to S^1_{\bullet}$.
\begin{rem} \label{thh-may  ss with coeff} 
Given a cofibrant decreasingly filtered commutative monoid $I$ and a cofibrant $I$-module $M$ in the category $\Sp^{\mathbb{N}^{\op}}$, we can also construct a topological Hochschild-May  spectral sequence with coefficients and prove an analogous  fundamental theorem of the May filtration with coefficients. This produces a spectral sequence 
\begin{equation} \label{thh may ss coeff} E^1_{*,*} = E_*(THH(E_0^*(I),E_0^*(M) ))\Rightarrow E_* (THH(I_0, M_0)) \end{equation} 
with the same differential convention. The spectral sequence strongly converges when $\pi_k(I_i)\cong 0$ for $k<i$ and $\pi_k(M_i)\cong 0$ for $k<i$. Notably, the spectral sequence is not multiplicative unless $M$ is also a decreasingly filtered commutative ring spectrum and it is an algebra over the decreasingly filtered commutative monoid $I$. See \cite[Thm. A.3.3]{thhmay} for details. 
\end{rem} 
\subsection{Decreasingly filtered commutative monoids in spectra}
In order to compute with the topological Hochschild-May  spectral sequence, it is necessary to have examples of cofibrant decreasingly filtered commutative monoids in $\Sp$. In \cite[Thm. 4.2.1]{thhmay}, the author and A. Salch constructed a large class of examples by constructing a specific model for the Whitehead tower associated to cofibrant connective commutative ring spectrum. This model for the Whitehead tower of $j$ is used to compute mod $(p,v_1)$ homotopy of $THH(j)$ in the subsequent section, so we recall the statement of the theorem. 
\begin{thm}[Theorem 4.2.1 \cite{thhmay}] \label{filt} Let $R$ be a cofibrant connective commutative monoid in $\Sp$, then there is an associated cofibrant decreasingly filtered  commutative monoid in $\Sp$, 
\[ \xymatrix{ \dots \ar[r] &  \tau_{\ge 2}R  \ar[r] & \tau_{\ge 1}R  \ar[r]&  \tau_{\ge 0}R    }\] 
where 
$\pi_k(\tau_{\ge n}R)\cong \pi_k(R)$ for $k\ge n$ and $\pi_k(\tau_{\ge n}R)\cong 0$ for $k<n$, equipped with structure maps 
\[ \rho_{i,j}\co\tau_{\ge i}R \wedge \tau_{\ge j}R  \rightarrow \tau_{\ge i+j}R \]
satisfying commutativity, associativity, unitality, and compatibility encoded in the structure of a cofibrant object $\tau_{\ge \bullet}R$ in $\Comm \Sp^{\mathbb{N}^{\op}}$.
\end{thm}
\begin{rem} It is well known that a model for the Whitehead tower as an object in $\Sp^{\mathbb{N}^{\op}}$ exists. The main thrust of the proof in \cite[Thm. 4.2.1]{thhmay} is that $\tau_{\ge \bullet}R$ can be built with multiplicative structure; i.e. $\tau_{\ge \bullet}R$ can be constructed as a cofibrant object in $\Comm \Sp^{\mathbb{N}^{\op}}$. 
\end{rem}

\begin{exm} \label{exm J}Fix a prime $p\ge5$ and let $q$ be a prime power that topologically generates $\mathbb{Z}_p^{\times}$. Recall that $j$ is a cofibrant replacement in $\Comm \Sp$ for the connective commutative ring spectrum $K(\mathbb{F}_q)_p$. Theorem \ref{filt} produces a decreasingly filtered commutative monoid $\tau_{\ge \bullet} j$ in $\Sp$. We will write $\mathbb{J}$ in place of $\tau_{\ge \bullet} j$ for the sake of brevity. 
The associated graded $E_0^*\mathbb{J}$ is additively  equivalent to 
\[ H\pi_0j\vee\Sigma^{2p-3}H\pi_{2p-3}j\vee \Sigma^{4p-5}H\pi_{4p-5}j\vee ... \] 
or more succinctly $H\pi_*(j)$. As a commutative ring spectrum, $E^0_*\mathbb{J}$ is given by taking iterated square-zero extensions.  The homotopy groups $\pi_*(E_0^*\mathbb{J})$ are isomorphic to  $\pi_*(j)$ as graded rings, but $E_0^*\mathbb{J}$ is a generalized Eilenberg-Maclane spectrum. In other words, we have killed off all the Postnikov $k$-invariants. After smashing with $S/p$, there is an equivalence
\[ S/p \wedge E_0^*\mathbb{J}\simeq H\mathbb{F}_p \vee \bigvee_{i\ge 1} \Sigma^{(2p-2)i-1}H\mathbb{F}_p\vee \Sigma^{(2p-2)i}H\mathbb{F}_p \]  
and $S/p \wedge E_0^*\mathbb{J}$ is naturally a $S/p\wedge H\pi_0j$ algebra, or in other words an $H\mathbb{F}_p$-algebra. 

We compute the topological Hochschild-May spectral sequence \eqref{thh may ss} where $X_{\bullet}$ is a point 
\begin{align}\label{trivial topological Hochschild-May}
\pi_*(S/p \wedge E_0^*\mathbb{J}) \Rightarrow \pi_*(S/p\wedge  j). 
\end{align}
The spectral sequence collapses for bidegree reasons as is visible from the pattern indicated by Figure \ref{sseq:example figure}. There is therefore an additive isomorphism 
\[ \pi_*(S/p \wedge E_0^*\mathbb{J})\cong P(v_1)\otimes E(\alpha_1) \] 
where in the $E^{\infty}$-page the bidegrees are 
\begin{align*}
| \alpha_1v_1^{k-1}|=&((2p-2)k-1,(2p-2)k-1) \text{ and } \\
|v_1^k|=&((2p-2)k,(2p-2)k-1)
\end{align*}
 where the first coordinate is the topological degree and the second coordinate is the May filtration. 
\begin{figure} \label{sseq:example figure} 
\begin{center}
\tiny
\[ 
\begin{sseq}[entrysize=.5cm,ylabelstep=2,ylabels={0;;2p-3;;4p-5;;6p-7;;8p-9;;10p-11},xlabelstep=3,xlabels={0;;;2p-3;;;4p-5;;;6p-7;;;8p-9;;;10p-11;;;}]{18}{12}
\ssdrop{1}
\ssmove{3}{2}
\ssdrop{\alpha_1}
\ssmove{1}{0}
\ssdrop{v_1}
\ssmove{2}{2}
\ssdrop{\alpha_1v_1}
\ssmove{1}{0}
\ssdrop{v_1^2}
\ssmove{2}{2}
\ssdrop{\alpha_1v_1^2}
\ssmove{1}{0}
\ssdrop{v_1^3}
\ssmove{2}{2}
\ssdrop{\alpha_1v_1^3}
\ssmove{1}{0}
\ssdrop{v_1^4}
\ssmove{2}{2}
\ssdrop{\alpha_1v_1^4}
\ssmove{1}{0}
\ssdrop{v_1^5}
\ssmove{1}{1}
\ssdrop{\dots}
\end{sseq}
\]
\end{center}
\caption{The $E^{\infty}_{s,t}$-page of the $S/p$ topological Hochschild-May spectral sequence for $p\ge 3$ for $s\le 4p^2$ and all $t$. Here we are simply tensoring with a point and our decreasingly filtered commutative ring spectrum is $\mathbb{J}$. 
}\label{sseq:example figure} 
\end{figure}

The ring structure on $\pi_*(S/p \wedge E_0^*\mathbb{J})$, however, is not the same as $\pi_*(S/p\wedge j)$ and instead each of the elements $\alpha_1v_1^k$ and $v_1^k$ are indecomposable for $k\ge 0$ and all the products 
$\alpha_1 v_1^n \cdot v_1^m $, $\alpha_1 v_1^n \cdot \alpha_1 v_1^k$ and $v_1^\ell \cdot v_1^m$ 
are trivial when $m>0$ and $\ell>0$. This is clear because there is an isomorphism of graded rings $\pi_*(E_0^*\mathbb{J})\cong \pi_*j$ and the products of elements $\alpha_j\cdot \alpha_k$, where  $\alpha_j\in \pi_{(2p-2)j-1}j$, are trivial for degree reasons and the product on $\pi_*(S/p\wedge E_0^*\mathbb{J})$ is induced by the product on $E_0^*\mathbb{J}$. It is also clear because $\pi_*(E_0^*\mathbb{J})$ is a bigraded ring and the products indicated must all be trivial for bidegree reasons. 

Therefore, there must be hidden multiplicative extensions in the spectral sequence \eqref{trivial topological Hochschild-May}. 
In particular, the existence of the spectral sequence implies that there is a filtration of the graded ring 
\[ \pi_*(S/p\wedge j)=P(v_1)\otimes E(\alpha_1),\] 
given by the May filtration, whose associated graded is the $E^1\cong E^{\infty}$-page of the spectral sequence. 

We explicitly describe the multiplicative filtration of the graded ring $P(v_1)\otimes E(\alpha_1)$ whose associated graded is the $E^{\infty}$-page of the spectral sequence.  It can be written explicitly as 
\begin{align}\label{filtration by powers of ideal}
\dots  \subset (\alpha_1 ,v_1)^3 \subset  (\alpha_1,v_1)^2  \subset  \dots   \subset  (\alpha_1,v_1)^2    \subset (\alpha_1, v_1)   \subset   \dots   \subset   (\alpha_1, v_1)   \subset  P(v_1)\otimes E(\alpha_1)
\end{align}
where the redundant copies of $(\alpha_1,v_1)^k$ are simply included so that the filtration matches the one coming from the Whitehead filtration on $j$. More precisely, writing $F_\bullet(P(v_1)\otimes E(\alpha_1))$ for this filtration, then there are equalities
\begin{align}\
\label{Fi def} F_0(P(v_1)\otimes E(\alpha_1))=&P(v_1)\otimes E(\alpha_1)=\pi_*(S/p \wedge E_0^*\mathbb{J}), \text{ and  }\\  
\nonumber F_{s}(P(v_1)\otimes E(\alpha_1))=&(\alpha_1,v_1)^{k(s)}  \text{ when }(2p-2)(k(s)-1)\le s \le(2p-2)k(s)-1
\end{align}
when $s\ge 1$. This filtration is clearly multiplicative with multiplication maps 
\[ F_s(P(v_1)\otimes E(\alpha_1)) \otimes F_t(P(v_1)\otimes E(\alpha_1)) \to F_{s+t}(P(v_1)\otimes E(\alpha_1)) \]
given by the composite maps 
\begin{align}\label{mult filt}
 (\alpha_1 ,v_1)^{j(s)} \otimes (\alpha_1,v_1)^{k(t)} \to (\alpha_1, v_1)^{j(s)+k(s)} \hookrightarrow  (\alpha_1, v_1)^{\ell(s+t)}
 \end{align}
where either 
\begin{align*}
	1 \le k(t) \le  j(s) \le \ell(s+t)\le j(s)+k(t) &\text{ or } \\
	1 \le j(s) <  k(t)\le \ell(s+t) \le j(s)+k(t)&
\end{align*} 
and the map 
\[ (\alpha_1 ,v_1)^n \otimes (\alpha_1 ,v_1)^m \to  (\alpha_1 ,v_1)^{n+m}\]
is the standard multiplication map for all $n,m\ge 1$. The multiplication maps when $j(s)=0$ or $k(s)=0$ are simply the right and left $P(v_1)\otimes E(\alpha_1)$ module structure maps for $P(v_1)\otimes E(\alpha_1)$-bimodule $(\alpha_1 ,v_1)^{k(s)} $. These multiplication maps satisfy the necessary commuting diagrams by construction of the multiplicative topological Hochschild-May spectral sequence. The associated graded of this multiplicative filtration is 
\[ \bigoplus_{i\ge 0} F_i(P(v_1)\otimes E(\alpha_1))/F_{i+1}(P(v_1)\otimes E(\alpha_1)) =S/p_*E_0^*\mathbb{J}.\]
To see this multiplicatively, note that $F_s(P(v_1)\otimes E(\alpha_1))/F_{s+1}(P(v_1)\otimes E(\alpha_1))$ is only nontrivial for $s\ge 1$ when $s=(2p-2)k-1$ for some $k$. Consequently, all products of elements of the form $\alpha v_1^j$, $v_1^k$  for $j\ge 0$ and $k\ge 1$ are trivial. 

From Figure \ref{sseq:example figure} it is also clear that all products of $\alpha_1v_1^j$ and $v_1^k$ for $j\ge 0$ and $k\ge 1$ in $S/p_*(E_0^*\mathbb{J})$ must be trivial for bidegree reasons and yet there is room for multiplicative extensions. Each of these multiplicative extensions must occur by strong convergence of the multiplicative topological Hochschild-May spectral sequence and the known multiplication on $S/p_*j\cong P(v_1)\otimes E(\alpha_1)$. More precisesly, there are multiplicative extensions 
\begin{align*} 
\alpha_1v_1^k \cdot v_1^{m}= & \alpha_1v_1^{k+m} \\
v_1^k \cdot v_1^{m}= & v_1^{k+m} \\
\end{align*}
and for bidegree reasons there is no room for further multiplicative extensions.
We use this fact in Lemma \ref{lem: homology of ag},  Remark \ref{HM remark},  and Proposition \ref{prop d_1}. 
\end{exm}

\section{Topological Hochschild homology of $j$ mod $(p,v_1)$}\label{section 3}
In Section \ref{section 2}, we reviewed the construction that takes a decreasingly filtered commutative monoid $I$ in $\Sp$ as input and produces a May-type spectral sequence
\[ E^1_{s,t}=E_{s,t}(THH(E_0^*I))\Rightarrow E_{s}(THH(I_0)) \] 
for any connective generalized homology theory $E$, which we we call the $E$ topological Hochschild-May  spectral sequence. Here, the second grading keeps track of the May filtration. Also, in Section \ref{section 2} we produced a Whitehead-type decreasingly filtered commutative monoid in $\Sp$, denoted $\mathbb{J}$, associated to a cofibrant commutative ring spectrum  model for $p$-complete connective image of J, which we denote $j$. We therefore have a spectral sequence 
\begin{equation} \label{Ethhmayj} E^1_{s,t}=E_{s,t}(THH(E_0^*\mathbb{J}))\Rightarrow E_{s}(THH(j)). \end{equation}
The purpose of this section is to compute this spectral sequence in the case $E=V(1)$. 
\subsection{ Computing the $H\mathbb{F}_p$ topological Hochschild-May   spectral sequence} 
In the case $E=H\mathbb{F}_p$, the input of the spectral sequence is calculable, and the output is already known due to Angeltveit-Rognes \cite{MR2171809}. The $H\mathbb{F}_p$ topological Hochschild-May  spectral sequence computation, therefore, will allow us to determine differentials in the $V(1)$ topological Hochschild-May  spectral sequence that are also detected in the $H\mathbb{F}_p$ topological Hochschild-May  spectral sequence. To begin, let us recall the computation of Angeltveit-Rognes. 
\begin{thm}[Angeltveit-Rognes {\cite[Prop. 7.12, Thm. 7.15]{MR2171809}}] \label{HFpj}
There is an isomorphism 
\[ {H\mathbb{F}_p}_*(j)\cong P(\tilde{\xi}_1^p,\tilde{\xi}_2,\bar{\xi}_3,...)\otimes E(\tilde{\tau}_2,\bar{\tau}_3, ...)\otimes E(b) \cong (\mathcal{A}//A(1))_*\otimes E(b)\] 
where all the elements in $(\mathcal{A}//A(1))_*$ besides $\tilde{\tau}_2$, $\tilde{\xi}_1^p$, and $\tilde{\xi}_2$, and $b$ have the usual $\mathcal{A}_*$-coaction and the coaction on the remaining elements $\tilde{\tau}_2$, $\tilde{\xi}_1^p$, $\tilde{\xi}_2$, and $b$ is given as\\ 
\begin{align*}
\psi(b)=& 1\otimes b \\
\psi(\tilde{\xi}_1^p)= & 1\otimes \tilde{\xi}_1^p -\tau_0\otimes b + \bar{\xi}_1^p\otimes 1 \\
\psi(\tilde{\xi}_2)= &1\otimes \tilde{\xi}_2+\bar{\xi}_1\otimes \tilde{\xi}_1^p +\tau_1\otimes b +  \bar{\xi}_2\otimes 1 \\
\psi(\tilde{\tau}_2)= & 1\otimes \tilde{\tau}_2 + \bar{\tau}_0\otimes \tilde{\xi}_2 +\bar{\tau}_1\otimes \tilde{\xi}_1^p  - \tau_1 \tau_0\otimes b  + \bar{\tau}_2\otimes 1  .
\end{align*}

There is also an isomorphism
\[ {H\mathbb{F}_p}_*(THH(j))\cong {H\mathbb{F}_p}_*(j)\otimes E(\sigma \tilde{\xi}_1^p,\sigma \tilde{\xi}_2)\otimes P(\sigma \tilde{\tau}_2)\otimes \Gamma(\sigma b) \]
of $\mathcal{A}_*$-comodules and ${H\mathbb{F}_p}_*(j)$-algebras. The $\mathcal{A}_*$-coaction is given by the formula
\begin{equation}\label{coaction sigma} \psi(\sigma x)=(1\otimes \sigma)\circ \psi(x) \end{equation}
and the previously stated coactions. 
\end{thm}

Note that Angeltveit-Rognes use a tilde over a symbol, for example $\tilde{x}$ to signify that the element has a different coaction then the coaction on $x$ or $\bar{x}$. We now want to compute the input of the spectral sequence. First, we note that as described in Example \ref{exm J}, $S/p\wedge E_0^*\mathbb{J}$ is an $H\mathbb{F}_p$ algebra and hence $V(1)\wedge E_0^*\mathbb{J}$ is also an $H\mathbb{F}_p$ algebra. It is known more generally that $THH(R)$ is an $R$ algebra when $R$ is a commutative ring spectrum, so $V(1)\wedge THH(E_0^*\mathbb{J})$ is a $V(1)\wedge E_0^*\mathbb{J}$-algebra and in particular an $H\mathbb{F}_p$-algebra. We can therefore apply the following lemma, which is well known and can be found in Ausoni--Rognes \cite[Lem. 4.1]{MR2928844}. 
\begin{lem} \label{prim} Let $M$ be an $H\mathbb{F}_p$-algebra. Then $M$ is equivalent to a wedge of suspensions of $H\mathbb{F}_p$, and the Hurewicz map 
$ \pi_*(M)\to  {H\mathbb{F}_p}_*(M) $ 
induces an isomorphism between $\pi_*(M)$ and the subalgebra of $\mathcal{A}_*$-comodule primitives contained in ${H\mathbb{F}_p}_*(M)$. 
\end{lem} 

Therefore, computing the subalgebra of comodule primitives in ${H\mathbb{F}_p}_*(V(1)\wedge THH(E_0^*\mathbb{J}))$ will suffice for computing the input of the $V(1)$ topological Hochschild-May  spectral sequence. 
\begin{lem}\label{lem: homology of ag}
There is an isomorphism of graded $\mathbb{F}_p$-algebras
\begin{equation}\label{Fp im j input} \pi_*(H\mathbb{F}_p\wedge E_0^*\mathbb{J}) \cong (A//E(0))_*\otimes S/p_*E_0^*\mathbb{J} \end{equation}
and the $H\mathbb{F}_p$ topological Hochschild-May  spectral sequence for $X_{\bullet}=*$, 
\[ \pi_*(H\mathbb{F}_p\wedge E_0^*\mathbb{J}) \Rightarrow \pi_*(H\mathbb{F}_p\wedge j ) \]
has differentials 
\begin{align*}
d_{2p-3}(\tau_1) \dot{=} &v_1, \\
d_{2p-3}(\bar{\xi}_1)\dot{=} & \alpha_1,\\
d_{2p-2}(\bar{\tau}_1v_1^k) \dot{=}  &v_1^{k+1}, \text{ and }\\
d_{2p-2}(\bar{\xi}_1v_1^k)\dot{=}& \alpha_1v_1^k
\end{align*}
for $k\ge 1$, as well as those differentials generated by the Leibniz rule, and there is an isomorphism of graded $\mathbb{F}_p$ algebras
\[ E_{*,*}^{\infty} = P(\bar{\xi}_1^p,\bar{\xi}_2,\dots )\otimes E(\bar{\tau}_2,\bar{\tau}_3, \dots )\otimes E(\alpha_1\bar{\xi}_1^{p-1})\cong {H\mathbb{F}_p}_*(j).\]
\end{lem}
\begin{proof} 
 Since $H\mathbb{Z}\wedge S/p \simeq H\mathbb{F}_p$, there is an equivalence
$ H\mathbb{F}_p\wedge E_0^*\mathbb{J} \simeq H\mathbb{Z} \wedge S/p \wedge E_0^*\mathbb{J}$
and since $S/p\wedge E_0^*\mathbb{J}$ is an $H\mathbb{F}_p$-algebra, 
there is an equivalence 
\[ H\mathbb{Z} \wedge S/p \wedge E_0^*\mathbb{J} \simeq (H\mathbb{Z} \wedge H\mathbb{F}_p )\wedge_{H\mathbb{F}_p} S/p\wedge E_0^*\mathbb{J} .\] 
By the K\"unneth isomorphism we produce the desired isomorphism \eqref{Fp im j input}. 

We then examine the topological Hochschild-May  spectral sequence  for $X_{\bullet}=*$, for which the abutment is known by Theorem \ref{HFpj}. In topological degrees $0< m <2p^2-2p-1$, we know the abutment $({H\mathbb{F}_p})_m(j)$ is trivial. This forces the differentials $d_{2p-3}(\bar{\xi}_1)=\alpha_1$ and $d_{2p-3}(\bar{\tau}_1)=v_1$ since there is no other way for these classes to be killed by differentials. After invoking the Leibniz rule, we compute
\[ E^{2p-2}_{*,*}=P(\bar{\xi}_1^p,\bar{\xi}_2,\ldots )\otimes E(\bar{\tau}_2,\bar{\tau}_3,\ldots )\otimes \mathbb{F}_p\{1,\bar{\tau}_1v_1^k,v_1^{k+1},\bar{\xi}_1v_1^k,\alpha_1v_1^k | k\ge 1\}\]
where all products of elements in the set $\{\bar{\tau}_1v_1^k,v_1^{k+1},\xi_1v_1^k,\alpha_1v_1^k | k\ge 1\}$ are trivial. These elements occur in pairs $\{\bar{\tau}_1v_1^k,v_1^{k+1}\}$ in bidegrees $((2p-2)(k+1)+1,(2p-2)k-1)$ and $((2p-2)k,(2p-2p)(k+1)-1)$ respectively and pairs $\{\bar{\xi}_1v_1^k,\alpha_1v_1^k\}$ in bidegrees $((2p-2)(k+1),(2p-2)k-1)$ and $((2p-2)(k+1)-1,(2p-2p)(k+1)-1)$ respectively. When we compare to the abutment, we see that in degree $m=(2p-2)k-1$ for $k\ge 1$ the dimension of the $\mathbb{F}_p$-vector space of the $E_{2p-2}$-page contributing to this degree is one more than the abutment. Similarly, in degree $m=(2p-2)k$ the dimension of the $\mathbb{F}_p$-vector space of the $E^{2p-2}$-page contributing to degree $(2p-2)k$ for $k\ge 1$ is two more than the abutment. Finally, in degree $m=(2p-2)k+1$ for $k\ge 1$ the dimension of the $\mathbb{F}_p$-vector space of the $E^{2p-2}$-page contributing to degree $(2p-2)k+1$ for $k\ge 1$ is one more than the abutment. We observe that the differential pattern stated is the only one that could possibly lead to the known abutment. 
Therefore, there must be differentials $d_{2p-2}(\bar{\tau}_1v_1^k)=v_1^{k+1}$ and $d_{2p-2}(\bar{\xi}_1v_1^k)=\alpha_1v_1^{k}$ for $k\ge 1$. 

Finally, we note that the $E^{\infty}$-page is concentrated on two lines with each element in $P(\bar{\xi}_1^p,\bar{\xi}_2,\dots )\otimes E(\bar{\tau}_2,\bar{\tau}_3, \dots)$ in May filtration zero and all elements divisible by $\alpha_1\bar{\xi}_1^{p-1}$ in May filtration $2p-3$. A hidden multiplicative extension would contradict the fact that the spectral sequence is multiplicative and known to strongly converge to $(H\mathbb{F}_p)_*(j)$. Consequently, there is no room for multiplicative extensions. 
\end{proof}
\subsubsection{The algebraic Hochschild-May spectral sequence}
A large number of differentials in the $H\mathbb{F}_p$ topological Hochschild-May  spectral sequence can be determined by algebraic means using the algebraic Hochschild-May spectral sequence \cite[Prop. 2.1]{150102547S}. We briefly recall the setup of the algebraic Hochschild-May spectral sequence in the specific case of interest. 

As noted in Example \ref{exm J}, there is a filtration of the graded ring $S/p_*(j)\cong P(v_1)\otimes E(\alpha_1)$ given by the May filtration whose associated graded is $S/p_*(E_0^*\mathbb{J})$.
Specifically, if we write $\mfilt(x)$ for the May filtration of an element $x\in P(v_1)\otimes E(\alpha_1)$, then $\mfilt(\alpha_1v_1^{k-1})=\mfilt(v_1^{k})=(2p-2)k-1$ for $k\ge 1$ and $\mfilt(1)=0$. By adjusting the grading by a linear transformation, the filtration of $S/p_*(j)$ makes more sense in the algebraic setting. As it stands, the grading is given by $|v_1^k|=((2p-2)k,(2p-2)k-1)$ and $|\alpha_1|=((2p-2)k-1,(2p-2)k-1)$ where the second grading is the May filtration grading. Since the topological Hochschild-May  spectral sequence has the grading convention 
\[ S/p_{s,t} (E_0^*\mathbb{J})\Rightarrow S/p_s( j) \]
we will adjust this so that the spectral sequence is of the form 
\[ S/p_{p,q}(E_0^*\mathbb{J}) \Rightarrow S/p_{p+q} (j )\] 
where $t=q$, $p+q=s$, and therefore $p=s-t$. This makes more sense algebraically because we can think of $S/p_{*,*}(E_0^*\mathbb{J})$ as a bigraded ring whose totalization $\Tot_*(S/p_{*,*}(E_0^*\mathbb{J})$ is additively isomorphic to $S/p_*(j)$. Consequently, in the remainder of this section we will use the grading convention $|v_1^k|=(1,(2p-2)k-1)$ and $|\alpha_1|=(0,(2p-2)k-1)$. 

We will write $F_i=\{x \in P(v_1)\otimes E(\alpha_1) : \mfilt(x)\ge i\}$. This produces a filtration of the chain complex associated to the cyclic bar complex 
\[
\xymatrix{
F_0 & \ar[l] F_0\otimes F_0 & \ar[l]  F_0\otimes F_0 \otimes F_0 & \ar[l] \ldots \\
F_1 \ar[u] & \ar[l] F_1\otimes F_0 +F_0 \otimes F_1 \ar[u] &  \ar[l] \sum_{i+j+k=1}F_i\otimes F_j\otimes F_k \ar[u] & \ar[l] \ldots \\
F_2 \ar[u] & \ar[l] \sum_{i+j=2} F_i\otimes F_j \ar[u] &  \ar[l] \sum_{i+j+k=2}F_i\otimes F_j\otimes F_k \ar[u] & \ar[l] \ldots \\
\vdots \ar[u] & \vdots \ar[u] &\vdots \ar[u] \\
}
\]
and the usual spectral sequence of a double chain complex produces a Hochschild-May spectral sequence 
\[ E_{s,p,q}^1=HH_{s,p,q}(E_0^*F_{\bullet}) \Rightarrow HH_{s,p+q}(P(v_1)\otimes E(\alpha_1)) \]
where the $s$ is the homological grading, $p$ internal grading, and $q$ is the May filtration grading, after applying the linear transformation described above. Here $HH_{s,p,q}(E_0^*F_{\bullet})$ is the Hochschild homology of the bigraded ring $E_0^*F_{\bullet}$, which agrees with $HH_s(\Tot E_0^*F_{\bullet})$, and $HH_{s,p+q}(P(v_1)\otimes E(\alpha_1))$ is Hochschild homology of the graded ring $P(v_1)\otimes E(\alpha_1)$. 

Since the filtration $F_{\bullet}$ is clearly complete, Hausdorff, and exhaustive, we know the spectral sequence converges conditionally by Boardman \cite[Thm. 9.2]{Boa98}. For each fixed $p$, the degree $p$ part of the filtration $F_{\bullet}$ is eventually trivial. Therefore, when we consider the spectral sequence as a tri-graded spectral sequence it is clear that the spectral sequence converges strongly by \cite[Thm. 7.1]{Boa98}. The differential convention, for $r\ge 0$, is 
\[ d_r: E^{s,p,q}_r \longrightarrow E^{s-1,p,q+r}_r \]
as can be easily determined by usual differential convention of a spectral sequence associated to a filtered complex. Since $E_0^*F_{\bullet}\cong S/p_*(E_0^*\mathbb{J})$, we have proven the following lemma. 
\begin{lem}\label{HHMay}
There is a strongly convergent Hochschild-May spectral sequence 
\begin{equation} \label{HHM} E_{s,p,q}^1=HH_{s,p,q}(S/p_*(E_0^*\mathbb{J})) \Rightarrow HH_{s,p+q}(S/p_*(j)). \end{equation}
where $s$ is the homological degree,  $p$ is the internal grading, and $q$ is the May filtration grading.
\end{lem}
Note that the Hochschild homology $HH_*(S/p_*(E_0^*\mathbb{J}) )$ is the Hochschild homology of a {\em bigraded} ring, but this agrees with the Hochschild homology of the totalization $\Tot_*S/p_*E_0^*\mathbb{J}$. It therefore suffices to consider the total degree of elements in $S/p_*(E_0^*\mathbb{J})$ to compute the $E^1$-page of the algebraic Hochschild-May spectral sequence. Note that there is an isomorphism of graded rings
\[S/p_*(E_0^*\mathbb{J})\cong \mathbb{F}_p[x_1,x_2,x_3,x_4,\ldots ]/(x_i x_j | 0\le i\le j<\infty)\] 
where the total degrees are given by $|x_{2i-1}|=(2p-2)i-1$ and $|x_{2i}|=(2p-2)i$. 
Next, we observe that $S/p_*(E_0^*\mathbb{J})$ can be written as a filtered colimit $\underset{r\to \infty}{\colim \hspace{.05in}} A_{2r}$ where 
\begin{equation}\label{Ar} A_{2r}=\mathbb{F}_p[x_1,x_2,x_3,x_4,\dots ,x_{2r} ]/(x_ix_j | 0\le i\le j<2r)\end{equation} 
with the same grading convention as above. 
Since Hochschild homology commutes with filtered colimits, it suffices to compute $HH_*(A_{2r})$. This computation is known in the ungraded context due to Geisser-Hesselholt \cite[Lem. 2.2]{MR2875843}. We recall their setup since much of the computation carries over to the graded context. We will write 
$B^{\cy}_{\bullet}(A_{2r})$ for the cyclic bar complex of $A_{2r}$ such that $B^{\cy}(A_{2r}):=|B^{\cy}_{\bullet}(A_{2r})|$ and $HH_*(A_{2r})\cong \pi_*B^{\cy}(A_{2r})$. Note that in the graded context it is necessary to define the structure maps in the cyclic bar complex with a sign that depends on the grading (the Koszul sign convention) as in Loday \cite[Sec. 5.3.2]{MR1600246}. Therefore, the maps in the graded context are the same as those in \cite[Sec. 2]{MR2875843} except for the maps
\begin{equation}\label{dn} d_n(a_0\otimes a_1 \otimes \dots \otimes a_n)=(-1)^{|a_n|(|a_0|+|a_1|+\dots |a_{n-1}|)}(a_n \cdot a_0 \otimes a_1 \otimes \dots \otimes a_{n-1}),\end{equation}
\begin{equation}\label{tm} t_m(a_0\otimes a_2 \otimes  \dots \otimes a_m)=(-1)^{|a_m|(|a_0|+|a_1|+\dots |a_{m-1}|)} (-1)^m(a_m \otimes a_0 \otimes \dots \otimes a_{m-1}). \end{equation}
We will use the notation $\bar{t}_m$ for the operator 
\begin{equation}\label{tm2} \bar{t}_m(a_0\otimes a_2 \otimes  \dots \otimes a_m)=(a_m \otimes a_0 \otimes \dots \otimes a_{m-1}). \end{equation}
Following \cite{MR2875843}, we write $\omega$ for a function $\omega\co \{0,1, \dots , m-1\} \to \{x_1,x_2, \ldots , x_{2r}\}$, which we call a word with letters $\{x_1,x_2, \ldots , x_{2r}\}$, and we write $[\omega]$ for the $C_{m}$ orbit of a word where $C_{m}$ acts by cyclic permutation. The length of the orbit $[\omega]$ will be called the period of $[\omega].$ 
By convention, the empty word will be written as $[0]$ and it has period $1$. There is a splitting of graded $\mathbb{F}_p$ vector spaces
\[ B^{\cy}_{\bullet} (A_{2r})=\bigoplus_{[\omega]}B^{\cy}_{\bullet} (A_{2r};[\omega ]) \]
where $B^{\cy}_{\bullet} (A_{2r};[\omega ])$, for $\omega = (x_{i_0},x_{i_1}, \dots x_{i_m})$ and $m\ge 1$, is the sub-cyclic graded $\mathbb{F}_p$ vector space of $B^{\cy}_{\bullet} (A_{2r})$ generated by 
$x_{i_1}\otimes \dots \otimes x_{i_m}$. Here a cyclic graded $\mathbb{F}_p$ vector space should be interpreted as a functor $\Lambda^{\op}\to \mathbb{F}_p\text{-mod}$ where $\Lambda$ is Connes' cyclic category and $\mathbb{F}_p\text{-mod}$ is the category of $\mathbb{F}_p$ vector spaces. In other words, we are freely adding in all degeneracy and cyclic operations acting on $x_{i_1}\otimes \dots \otimes x_{i_m}$. We will write 
\begin{equation}\label{not} HH_*(A_{2r};[\omega])\cong \pi_*|B^{\cy}(A_{2r};[\omega])|. \end{equation}
\begin{lem}
Let $A_{2r}$ be the graded $\mathbb{F}_p$-algebra defined in \eqref{Ar} and $HH_*(A_{2r};[\omega])$ be the summand of $HH_*(A_{2r})$ defined in \eqref{not}. 
Let $\omega=(x_{i_1},\dots ,x_{i_m})$ be a word with letters $\{x_1,x_2,\ldots, x_{2r}\}$ of length $m\ge 0$ and period $\ell \ge 1$. 
Then if $m=0$, $HH_0(A_{2r};[\omega])\cong \mathbb{F}_p\{1\}$. 
For $m\ge 1$, if 
$(m-1)\ell+(|x_{i_{m-\ell+1}}|+\ldots |x_{i_m}|)(|x_{i_1}|+\ldots +|x_{i_{m-\ell}}|)$ is even then 
$HH_{m-1}(A_{2r};[\omega])$ is a free $\mathbb{F}_p$ vector space of rank one generated by the cycle $x_{i_1}\otimes \dots \otimes x_{i_m}$ and $HH_{m}(A_{2r};[\omega])$ is a free $\mathbb{F}_p$ vector space of rank one generated by the cycle $N(x_{i_1}\otimes\dots \otimes x_{i_m})$ defined as
\[\sum_{0\le u<\ell}\bar{t}_ms_{m-1}t_{m-1}^u(x_{i_1}\otimes\dots \otimes x_{i_m}).\] 
For $m\ge 1$, if  $(m-1)\ell+(|x_{i_{m-\ell+1}}|+\ldots |x_{m}|)(|x_{i_1}|+\ldots +|x_{i_{m-\ell}}|)$ is odd, then 
\[HH_{m-1}(A_{2r};[\omega])\cong HH_{m}(A_{2r};[\omega])\cong 0.\] 
\end{lem}
\begin{proof}
We will adjust the proof of \cite[Lem. 2.2]{MR2875843} to the graded context. 
To compute $HH_*(A_{2r};[\omega])$ it suffices to compute $H_*(D_\bullet)$ where $D_{\bullet}$ is the associated normalized chain complex \cite[Thm. 8.3.8]{MR1269324}. When $m=0$, then $D_{\bullet}=\mathbb{F}_p\{1\}$ where $\mathbb{F}_p\{1\}$ a the free $\mathbb{F}_p$ vector space generated by $1$ concentrated in degree zero, so $H_*(D_{\bullet})\cong \mathbb{F}_p\{1\}$ as well. Now observe that the internal grading of all elements in $B^{\cy}(A_{2r};[\omega])$ are the same so the same method as \cite[Lem. 2.2]{MR2875843} applies in the graded context except that the boundary maps must be adjusted according to the grading. 

The chain complex computing $HH_*(A_{2r};[\omega])$ is concentrated in degree $m$ and $m-1$ with 
\[ D_m=\mathbb{F}_p\{\bar{t}_ms_{m-1}t_{m-1}^k(x_{i_1}\otimes \ldots x_{i_m})| 0\le k<\ell\}\] 
and 
\[D_{m-1}=\mathbb{F}_p\{t_{m-1}^k(x_{i_1}\otimes \ldots x_{i_m})| 0\le k<\ell\} \]
and differential $d$ given by the the usual graded Hochschild differential convention \cite[Sec. 5.3.2]{MR1600246}. 
We split into three cases. First, consider when $(m-1)\ell+(|x_{i_{m-\ell+1}}|+\ldots +|x_{i_m}|)(|x_{i_1}|+\ldots +|x_{i_{m-\ell}}|)$ is odd. 
Let $D^{\prime}$ be the chain complex concentrated in degrees $m$ and $m-1$ with $D^{\prime}_m= D^{\prime}_{m-1}=\mathbb{F}_p[C_{\ell}]$ and differential 
$d^{\prime}$ given by multiplication by $1-\tau$. 
Define $\alpha\co D^{\prime} \to D$ by 
\[ \alpha_m(\tau^u)= t_{m-1}^u(x_{i_1}\otimes \ldots \otimes x_{i_m}) \]
\[ \alpha_{m-1}(\tau^u)= \bar{t}_ms_{m-1}t_{m-1}^u(x_{i_1}\otimes \ldots \otimes x_{i_m})\]
and define 
$N=1+\tau+\ldots \tau^{\ell-1}.$
Note that, by a straightforward computation,
\begin{equation}\label{dN} \begin{array}{l} d(N(x_{i_1}\otimes \ldots \otimes x_{i_m}))=  \\ x_{i_1}\otimes \ldots x_{i_m}-(-1)^{(m-1)\ell+(|x_{i_{m-\ell+1}}|+\ldots +|x_{i_m}|)(|x_{i_1}|+\ldots +|x_{i_{m-\ell}}|)}x_{i_1}\otimes \ldots x_{i_m}
\end{array}
\end{equation}
where the sign on the last term is determined by \eqref{dn}.
Since $(m-1)\ell+(|x_{i_{m-\ell+1}}|+\ldots +|x_{i_m}|)(|x_{i_1}|+\ldots +|x_{i_{m-\ell}}|)$ is even by assumption, $\alpha$ is an isomorphism of chain complexes and $H_mD^{\prime}=\mathbb{F}_p\{N\}$ and $H_{m-1}D^{\prime}=\mathbb{F}_p\{1\}$ and otherwise $H_k(D^{\prime})=0$. Thus, $H_m(D)\cong\mathbb{F}\{N(x_{i_1}\otimes \ldots \otimes x_{i_m})\}$, $H_{m-1}(D)\cong \mathbb{F}_p\{x_{i_1}\otimes \ldots \otimes x_{i_m}\}$ and $H_k(D)\cong 0$ otherwise. (Note that in \cite[Lem. 2.2]{MR2875843} the signs of $\alpha_m(\tau^u)$ and $\alpha_{m-1}(\tau^u)$ seem to differ from ours significantly, however this is accounted for by our convention for the sign of $t_m^u$ in formula \eqref{tm}, which differs from Geisser-Hesselholt's convention. Our convention is standard for Hochschild homology of a graded ring.) 
 
When $(m-1)\ell+(|x_{i_{m-\ell+1}}|+\ldots +|x_{i_m}|)(|x_{i_1}|+\ldots +|x_{i_{m-\ell}}|)$ is odd and $\ell$ is odd, consider the chain complex $D^{\prime \prime}$ concentrated in degrees $m$ and $m-1$ with $D^{\prime \prime}_m=D^{\prime \prime}_{m-1}=\mathbb{F}_p[C_{\ell}]$ and boundary map $d^{\prime \prime}=1+\tau$. Then let $\beta\co D^{\prime \prime}\to D$ be defined by the formulas
\[ \beta_m(\tau^u)=(-1)^{u}t_{m-1}^u(x_{i_1}\otimes \ldots \otimes x_{i_m}) \]
\[ \beta_{m-1}(\tau^u)=(-1)^{u}\bar{t}_ms_mt_{m-1}^u(x_{i_1}\otimes \ldots \otimes x_{i_m}). \]
Since $(m-1)\ell+(|x_{i_{m-\ell+1}}|+\ldots+ |x_{i_m}|)(|x_{i_1}|+\ldots + |x_{i_{m-\ell}}|)$ is odd by assumption and $\ell$ is odd, the differential \eqref{dN} is compatible with $\beta$ and $d^{\prime \prime}$ and $\beta$ is an isomorphism of chain complexes. Since $1+\tau\co \mathbb{F}_p[C_{\ell}]\to \mathbb{F}_p[C_{\ell}]$ is an isomorphism when $\ell$ is odd, $H_k(D)\cong 0$ for all $k\ge 0$. 

When $(m-1)\ell+(|x_{i_{m-\ell+1}}|+\ldots |x_{i_m}|)(|x_{i_1}|+\ldots |x_{i_{m-\ell}}|)$ is odd and $\ell$ is even, consider the chain complex $D^{\prime \prime \prime}$ concentrated in degrees $m$ and $m-1$ with $D^{\prime \prime \prime}_m=D^{\prime \prime \prime}_{m-1}=\mathbb{F}_p[x]/(x^{\ell}+1)$ and boundary map $d^{\prime \prime \prime}=1+y$. Then let $\beta\co D^{\prime \prime \prime}\to D$ be the defined by the formulas
\[ \beta_m(y^u)=(-1)^{u}t_{m-1}^u(x_{i_1}\otimes \ldots \otimes x_{i_m}) \]
\[ \beta_{m-1}(y^u)=(-1)^{u}\bar{t}_ms_mt_{m-1}^u(x_{i_1}\otimes \ldots \otimes x_{i_m}). \]
Since $(m-1)(\ell)+(|x_{i_{m-\ell+1}}|+\ldots +|x_{i_m}|)(|x_{i_1}|+\ldots +|x_{i_{m-\ell}}|)$ is odd the differentials $d$ and  $d^{\prime \prime}$ are compatible with $\beta$. Also, $y^{\ell}=-1$ maps to 
\[(-1)^{\ell}\bar{t}_ms_mt_{m-1}^u(x_{i_1}\otimes \ldots \otimes x_{i_m}=(-1)^{\ell}(-1)^{(m-1)\ell+(|x_{i_{m-\ell+1}}|+\ldots |x_{i_m}|)(|x_{i_1}|+\ldots |x_{i_{m-\ell}}|)}1\otimes x_{i_1}\otimes \ldots \otimes x_{i_m}\]
 so since $\ell$ is even and $(m-1)\ell+(|x_{i_{m-\ell+1}}|+\ldots |x_{i_m}|)(|x_{i_1}|+\ldots |x_{i_{m-\ell}}|)$ is odd by assumption, $\beta$ is an isomorphism of chain complexes. Since $1+y\co \mathbb{F}_p[y]/(y^{\ell}+1)\to \mathbb{F}_p[y]/(y^{\ell}+1)$ is an isomorphism, $H_k(D)\cong 0$ for all $k\ge 0$. 
\end{proof}
From here on we will return to writing $\alpha_1v_1^k$ and $v_1^{k+1}$ for $x_{2k-1}$ and $x_{2k+2}$ respectively for $k\ge 0$. By an easy computation of the shuffle product in $HH_*(S/p_*E_0^*\mathbb{J})$, we see that
the shuffle product on $1\otimes \alpha_1$ is given by the formula 
\[ (1\otimes \alpha_1)^{\# n}=n\otimes \underset{n}{\underbrace{\alpha_1 \otimes \ldots \otimes \alpha_1}}.\]
Since $1\otimes \alpha_1$ is the usual cycle representative for $\sigma \alpha_1$ the shuffle product above makes $\gamma_n(\sigma \alpha_1)$ sensible notation for $1\otimes \underset{n}{\underbrace{\alpha_1 \otimes \ldots \otimes \alpha_1}}$. 

The differentials can be easily determined using the definition of a spectral sequence of a filtered chain complex. We therefore simply recall almost verbatim the discussion of the behavior of the differentials from \cite[Prop. 2.1]{150102547S} in our special case. 
\begin{lem}[cf. Salch {\cite[Prop. 2.1]{150102547S}}] \label{lem HHM}
The differential in the spectral sequence is computed on a class 
\[x \in HH_{*,*}(S/p_*(E_0^*\mathbb{J}), S/p_*(E_0^*\mathbb{J}))\] 
by computing a homogeneous cycle representative $y$ for $x$ in the standard Hochschild chain complex for $S/p_*(E_0^*\mathbb{J})$, lifting y to a homogeneous chain $\tilde{y}$ in the standard Hochschild chain complex for $A$, applying the Hochschild differential $d$ to $\tilde{y},$ then taking the image of $d(\tilde{y})$ in the standard Hochschild chain complex for $S/p_*(E_0^*\mathbb{J})$.
\end{lem}

\begin{cor}\label{key algebraic corollary 1}
The spectral sequence \eqref{HHM} collapses at the $E^2$-page with an additive isomorphism $E^2_{*,*}\cong E_{*,*}^{\infty} \cong E(\alpha_1)\otimes P(v_1)\otimes \Gamma(\sigma \alpha_1)\otimes E(\sigma v_1)$. 
\end{cor}
\begin{proof}
By Lemma \ref{lem HHM}, given cycles $N(\alpha_1^{\epsilon_1} v_1^{i_1} \otimes \alpha_1^{\epsilon_2}v_1^{i_2} \otimes  \ldots \alpha_1^{\epsilon_m}v_1^{i_m})$ and $\alpha_1^{\epsilon_0} v_1^{i_0}\otimes \alpha_1^{\epsilon_1}v_1^{i_1}\otimes \ldots \alpha_1^{\epsilon_m}v_1^{i_m}$ in the cyclic bar complex for $S/p_*E_0^*\mathbb{J}$, the differentials in the algebraic Hochschild-May spectral sequence are given by 
the following formulas. First, we compute
\[ 
 \begin{array}{rc}
	      d^1(N(\alpha_1^{\epsilon_1} v_1^{i_1} \otimes \alpha_1^{\epsilon_2}v_1^{i_2} \otimes  \ldots \otimes \alpha_1^{\epsilon_m}v_1^{i_m})) &= \\
	       \sum_{0\le u<\ell}d^1 (\bar{t}_ms_mt_{m-1}^u(\alpha_1^{\epsilon_1} v_1^{i_1} \otimes \alpha_1^{\epsilon_2}v_1^{i_2} \otimes  \ldots \otimes \alpha_1^{\epsilon_m}v_1^{i_m})) &= \\
	       \sum_{0\le u<\ell}\gamma_u d^1 (1\otimes \alpha_1^{\epsilon_{m-u+1}} v_1^{i_{m-u+1}} \otimes \alpha_1^{\epsilon_{m-u+2}}v_1^{i_{m-u+2}} \otimes  \ldots \otimes \alpha_1^{\epsilon_{m-u}}v_1^{i_{m-u}})&
	      \end{array}
\]
where $\gamma_u =(-1)^{m-1}(-1)^{(|\alpha_1^{\epsilon_{m-u+1}}v_1^{i_{m-u+1}}|+\ldots + |\alpha_1^{\epsilon_m}v_1^{i_m}|)(|\alpha_1^{\epsilon_1}v_1^{i_1}+\ldots + |\alpha_1^{\epsilon_{m-u}}v_1^{i_{m-u}}|)}$
and 
\[ \begin{array}{rc}	  
d^1 (1\otimes \alpha_1^{\epsilon_{m-u+1}} v_1^{i_{m-u+1}} \otimes \alpha_1^{\epsilon_{m-u+2}}v_1^{i_{m-u+2}} \otimes  \ldots \otimes \alpha_1^{\epsilon_{m-u}}v_1^{i_{m-u}}))&=  \\
 -  1\otimes \alpha_1^{\epsilon_{m-u+1}+\epsilon_{m-u+2} }v_1^{i_{m-u+1}+i_{m-u+2} } \otimes \alpha_1^{\epsilon_{m-u+3} } v_1^{i_{m-u+3} }\otimes \ldots \otimes  \alpha_1^{\epsilon_{m-u} }v_1^{ i_{m-u} } +     &  \\
	  1\otimes \alpha_1^{\epsilon_{m-u+1}}v_1^{i_{m-u+1}}\otimes \alpha_1^{\epsilon_{m-u+2}+\epsilon_{m-u+3}}v_1^{i_{m-u+2}+i_{m-u+3}}\otimes \alpha_1^{\epsilon_{m-u+4}}v_1^{i_{m-u+4}}\otimes \ldots \otimes \alpha_1^{\epsilon_{m-u} }v_1^{ i_{m-u} } &\\
	      + \ldots + & \\
	    (-1)^{m-1}  1\otimes \alpha_1^{\epsilon_{m-u+1}}v_1^{i_{m-u+1}}\otimes \ldots \alpha_1^{\epsilon_{m-u-1}+\epsilon_{m-u} }v_1^{i_{m-u-1} +i_{m-u}} &\\
	    + {E_u}
	      \end{array}
\]
where $\sum_{0\le u<\ell}E_u=0$. Here $\epsilon_i\in \{0,1\}$ for all  $0\le i\le m$, $m\ge 2$, and $\ell$ is the length orbit of the $C_m$ action. 

Consequently, all of the cyclic bar complex cycles of the form $N(\alpha_1^{\epsilon_1} v_1^{i_1} \otimes \alpha_1^{\epsilon_2}v_1^{i_2} \otimes  \ldots \otimes \alpha_1^{\epsilon_m}v_1^{i_m})$ where $\epsilon_i=1$ for all $0\le i\le m$ are $d^1$-cycles. When in addition $i_j=0$ for all $0\le j\le m$, then  
\[ d^1(N(\underset{m}{\underbrace{\alpha_1\otimes \ldots \otimes \alpha_1}}))= \underset{m}{\underbrace{\alpha_1\otimes \ldots \alpha_1}}+(-1)^m(-1)^{|\alpha_1|(|\alpha_1|(m-1))}\underset{m}{\underbrace{\alpha_1\otimes \ldots \alpha_1}}\]
where if $m$ is even, $|\alpha_1|(|\alpha_1|(m-1))$ is odd, and if $m$ is odd, then $|\alpha_1|(|\alpha_1|(m-1))$ is even so in either case $(-1)^m(-1)^{|\alpha_1|(|\alpha_1|(m-1)}=-1$ and the differential is zero. It is also easy to observe from the formula above that the element $N(\alpha_1\otimes \ldots \otimes \alpha_1)$ cannot be a boundary. If $\epsilon_i=1$ for $0\le i\le m$ except $\epsilon_k=0$ and $i_j=0$ for all $0\le j\le m$ except $i_k=1$, then we compute 
$d^1(N(v_1\otimes \alpha_1\otimes \ldots \alpha_1))=0$. All other cyclic bar complex cycles of the form $N(\alpha_1^{\epsilon_1} v_1^{i_1} \otimes \alpha_1^{\epsilon_2}v_1^{i_2} \otimes  \ldots \otimes \alpha_1^{\epsilon_m}v_1^{i_m})$ are either the source of a differential or they are boundaries. 

Similarly, there are differentials
\[ 
 \begin{aligned}
d^1(\alpha_1^{\epsilon_0} v_1^{i_0}\otimes \alpha_1^{\epsilon_1}v_1^{i_1}\otimes \ldots \alpha_1^{\epsilon_m}v_1^{i_m})= \\\alpha_1^{\epsilon_0+\epsilon_1} v_1^{i_0+i_1}\otimes \alpha_1^{\epsilon_2}v_1^{i_2} \otimes \ldots \alpha_1^{\epsilon_m}v_1^{i_m}  -   
\alpha_1^{\epsilon_0}v_1^{i_0}\otimes \alpha_1^{\epsilon_1+\epsilon_2}v_1^{i_1+i_2}\otimes \ldots \otimes \alpha_1^{\epsilon_m}v_1^{i_m} + \ldots +  \\
(-1)^{m}(-1)^{|\alpha_1^{\epsilon_m}|(|\alpha_1^{\epsilon_0}v_1^{i_0}|+\ldots + |\alpha_1^{m-1}v_1^{i_{m-1}}|)}\alpha_1^{\epsilon_0+\epsilon_m}v_1^{i_0+i_m}\otimes \alpha_1^{\epsilon_1}v_1^{i_1} \otimes \ldots \alpha_1^{\epsilon_{m-1}}v_1^{i_{m-1}}
\end{aligned}
\]
where $\epsilon_i\in \{0,1\}$ for all  $0\le i\le m$ and $m\ge 2$, respectively. 

Consequently, if $\sum_{k=0}^m\epsilon_k\le m$ and $\sum_{k=0}^{m}i_k>2$ then $\alpha_1^{\epsilon_0} v_1^{i_0}\otimes \alpha_1^{\epsilon_1}v_1^{i_1}\otimes \ldots \alpha_1^{\epsilon_m}v_1^{i_m}$ is the source of a differential. Also, if $\sum_{k=1}^{m}i_k\ge 2$, then $\alpha_1^{\epsilon_0} v_1^{i_0}\otimes \alpha_1^{\epsilon_1}v_1^{i_1}\otimes \ldots \otimes \alpha_1^{\epsilon_m}v_1^{i_m}$ is necessarily the boundary of a differential. The only elements of the form $\alpha_1^{\epsilon_0} v_1^{i_0}\otimes \alpha_1^{\epsilon_1}v_1^{i_1}\otimes \ldots \otimes \alpha_1^{\epsilon_m}v_1^{i_m}$ that are not the boundary or source of a differential are the elements of the form
$ \alpha_1^{\epsilon_0}v_1^{i_0}\otimes \alpha_1\otimes \ldots \otimes\alpha_1 $
and 
$\alpha_1^{\epsilon_0}v_1^{i_0}\otimes v_1\otimes \ldots \otimes \alpha_1,$ which survive to become $\alpha_1^{\epsilon_0}v_1^{i_0}\gamma_m\sigma b$ and $\alpha_1^{\epsilon_0}v_1^{i_0}\sigma v_1\gamma_m\sigma b$ respectively in the abutment. Therefore, there is an isomorphism $E^2\cong E^{\infty}$ and an additive isomorphism $E_{*,*}^{\infty} \cong P(v_1)\otimes E(\alpha_1)\otimes E(\sigma v_1)\otimes \gamma(\sigma \alpha_1)$. 
\end{proof}
\begin{figure}
\begin{center}
\tiny
\[  
\begin{sseq}[entrysize=.9cm,ylabelstep=1,ylabels={0;;;;2p-3;;;;4p-5;;;;6p-7;;;;8p-9},xlabelstep=1,xlabels={0;;;2p-3;;;;4p-5;;;;6p-7;;;;}]{14}{14}
\ssdrop{1}
\ssmove{3}{4}
\ssdrop{\alpha}
\ssmove{1}{0}
\ssdrop{v}
\ssdrop{\gamma_1}
\ssmove{1}{0}
\ssdrop{z}
\ssmove{2}{4}
\ssdrop{\alpha v }
\ssmove{0}{-1}
\ssdrop{\alpha\gamma_1 }
\ssmove{1}{1}
\ssdrop{v^2}
\ssmove{0}{-1}
\ssdrop{\gamma_2}
\ssdrop{v \gamma_1}
\ssdrop{\alpha  z}
\ssmove{1}{0}
\ssdrop{z \gamma_1}
\ssdrop{v z}
\ssmove{2}{5}
\ssdrop{\alpha v^2}
\ssmove{0}{-1}
\ssdrop{\alpha v  \gamma_1 }
\ssmove{0}{-1}
\ssdrop{\alpha \gamma_2}
\ssmove{1}{2}
\ssdrop{v^3}
\ssmove{0}{-1}
\ssdrop{\alpha v z}
\ssdrop{v^2\gamma_1}
\ssmove{0}{-1}
\ssdrop{v\gamma_2}
\ssdrop{\gamma_3}
\ssdrop{\alpha z \gamma_1}
\ssmove{1}{0}
\ssdrop{z\gamma_2}
\ssdrop{v z \gamma_1}
\ssmove{0}{1}
\ssdrop{v^2z}
\ssmove{0}{2}
\ssdrop{\dots}
\end{sseq}
\]
\end{center}
\caption{The $E^{\infty}_{*,*}$-page of the Hochschild-May spectral sequence in a range. Here we simply write $\alpha:=\alpha_1$, $v:=v_1$, $\gamma_i=\gamma_i(\sigma \alpha_1)$ and $z=\sigma v_1$. 
}\label{sseq:HHM} 
\end{figure}

\begin{remark}\label{HM remark}
It will also be useful to understand the multiplicative extensions in the spectral sequence \eqref{HHM} explicitly. We include Figure \ref{sseq:HHM} to facilitate this. Note that, as a ring, the $E^{\infty}$-page is isomorphic to 
\[ \left ( E(\sigma v_1)\otimes \Gamma(\sigma \alpha_1)\otimes P(\alpha_1v_1^k,v_1^{k+1}: j,k\ge 0) \right )/\left ((\alpha_1v_1^j)^2,(v_1^{k+1})(v_1^{j+1}),(\alpha_1v_1^j)(v_1^{k+1})  \text{ for all } j,k\ge 0\right )\]

We can describe the multiplicative filtration of the graded ring $P(v_1)\otimes E(\alpha_1)\otimes E(\sigma v_1)\otimes \Gamma(\sigma \alpha_1)$  coming from filtering the abutment as 
\[  \dots    \subset G_2\subset  G_1 \subset G_0=P(v_1)\otimes E(\alpha_1)\otimes E(\sigma v_1)\otimes \Gamma(\sigma \alpha_1)\] 
where 
\begin{align*}
G_s=\{ x \in G_0 : \text{mfilt}(x)\ge s\}
\end{align*}
for $s\ge 1$. Since $\text{mfilt}(xy)\ge \text{mfilt}(x)+\text{mfilt}(y)$ we observe that this is a multiplicative filtration. By inspection, the associated graded of this filtration is exactly 
\[ \bigoplus_{i\ge 0} G_{i}/G_{i+1}= E^{\infty}_{*,*,*}.\]
Trivial products between $w,z\in E_0^*G_{\bullet}$ arise when $\text{mfilt}(zwy)>\text{mfilt}(z)+\text{mfilt}(w)$. Suppose $z$ and $w$ are nontrivial and $zw$ is nontrivial in $G_0$. Then $\text{mfilt}(zw)>\text{mfilt}(z)+\text{mfilt}(w)$ exactly when 
\[ z,w \in \{ \alpha_1v_1^kx ,v_1^{k+1}x : k\ge 0, x\in E(\sigma v_1)\otimes \Gamma(\sigma \alpha_1) \}.\] 
Consequently, there are trivial products 
\[v_1^{j+1}x \cdot v_1^ky=0 \in E^{\infty}_{*,*,*}, \text{ and }\]
\[\alpha_1v_1^{j}x \cdot v_1^ky=0 \in E^{\infty}_{*,*,*}\]
for $j\ge 0$ and $k\ge 1$. We know, however, that in the abutment these products are nontrivial. We therefore observe that there must be multiplicative extensions 
\begin{align} 
\label{mult exts in HMSS} \alpha_1v_1^j x \cdot v_1^{k+1}y = \alpha_1 v_1^{j+k}xy \text{ for } j,k\ge 0,\\
\nonumber v_1^{j} x\cdot v_1^ky = v_1^{j+k}xy \text{ for } j+k\ge 1 
\end{align}
for all $x,y\in E(\sigma v_1)\otimes \Gamma(\sigma \alpha_1)$.  In particular, we know that the abutment is isomorphic to 
\[ HH_*(E(\alpha_1)\otimes P(v_1))\cong E(\alpha_1)\otimes P(v_1)\otimes  \Gamma (\sigma \alpha_1)\otimes E(\sigma v_1) \]
as  bi-graded $\mathbb{F}_p$-algebras by \cite[Prop. 2. 1]{MR1209233}. Therefore, there cannot be further multiplicative extensions. 
\end{remark}

\subsubsection{The $H\mathbb{F}_p$ topological Hochschild-May  spectral sequence}
We now observe that there is an equivalence of commutative ring spectra 
\[E_0^*\mathbb{J}\simeq H\mathbb{Z}_{(p)} \wedge (S\vee \bigvee_{k\ge 1} \Sigma^{(2p-2)k-1}S/p^{\nu_p(k)+1}) \]
where the commutative ring spectrum structure on $(S\vee \bigvee_{k\ge 1} \Sigma^{(2p-2)k-1}S/p^{\nu_p(k)+1})$ is given by iterating the trivial square-zero extension construction. For brevity, we will write $\overline{E_0^*\mathbb{J}}$ for the commutative ring spectrum $(S\vee \bigvee_{k\ge 1} \Sigma^{(2p-2)k-1}S/p^{\nu_p(k)+1})$. This immediately implies that 
\[ THH(E_0^*\mathbb{J})\simeq THH(H\mathbb{Z}_{(p)}) \wedge THH(\overline{E_0^*\mathbb{J}})  \]
since $S^1_{\bullet}\otimes (-)$ is a left adjoint and therefore commutes with coproducts (smash products) of commutative ring spectra. Before computing $H_*THH(E_0^*\mathbb{J})$, we need to prove a lemma about how the B\"okstedt spectral sequence interacts with the May filtration. We will use the notation $E_0^*|\mathcal{M}^{S^1_{\bullet}}(\mathbb{K})|$ for the associated graded of the cofibrant decreasingly filtered commutative ring spectrum $|\mathcal{M}^{S^1_{\bullet}}(\mathbb{K})|$ constructed as the realization of the simplicial decreasingly filtered commutative monoid in spectra $|\mathcal{M}^{S^1_{\bullet}}(\mathbb{K})|$ where $\mathbb{K}$ is a cofibrant decreasingly filtered commutative monoid in spectra, see Definition 3.3.3 \cite{thhmay} for further explanation of this notation. 
\begin{lem}\label{Bok-May}
Let $\mathbb{K}$ be a decreasingly filtered commutative ring spectrum as in Definition \ref{def:DFO}. Then the differentials in the B\"okstedt spectral sequence
\[ HH_*(E_0^*\mathbb{K})\Rightarrow H_*(THH(E_0^*\mathbb{K}))\]
respect the May filtration. 
\end{lem}
\begin{proof}
By the fundamental theorem of the May filtration \cite[Thm. 3.3.10]{thhmay}, there is an equivalence of commutative ring spectra 
\[ E_0^*|\mathcal{M}^{S^1_{\bullet}}(\mathbb{K})|\simeq THH(E_0^*\mathbb{K}).\]
We can filter $E_0^*\mathbb{K}=\bigvee_{i\ge 0} \mathbb{K}_i/ \mathbb{K}_{i+1}$ using the trivial filtration 
\[ E_0^*\mathbb{K} \leftarrow \bigvee_{i\ge 1} \mathbb{K}_i/ \mathbb{K}_{i+1} \leftarrow \bigvee_{i\ge 2} \mathbb{K}_i/ \mathbb{K}_{i+1} \leftarrow \ldots \]
and this is clearly a cofibrant decreasingly filtered commutative ring spectrum. There is therefore a topological Hochschild-May  spectral sequence
\begin{equation}\label{collapse} H_*(E_0^*|\mathcal{M}^{S^1_{\bullet}}(\mathbb{K})|) \Rightarrow H_*(THH(E_0^*\mathbb{K})),\end{equation}
which collapses. The $E^1$-page of the B\"okstedt spectral sequence is the chain complex $H_*(E_0^*\mathbb{K}^{\wedge t})$ with differential given by the standard Hochschild homology differential. There is a filtration of this chain complex by $H_*(\bigvee_{i\ge s} \mathbb{K}_i/ \mathbb{K}_{i+1})$, which produces a double complex and therefore a collapsing spectral sequence as in \eqref{collapse}. If we filter this double complex by B\"okstedt filtration to produce a triple complex, then since the spectral sequence collapses in one direction, this spectral sequence is equivalent to the B\"okstedt spectral sequence. This implies that the B\"okstedt spectral sequence differentials for this particular spectrum $E_0^*\mathbb{K}$, must respect the May filtration. 
\end{proof}
\begin{prop}\label{Bok1} 
There is an isomorphism of $\mathcal{A}_*$-comodules
\[ H_*(THH(E_0^*\mathbb{J}))\cong (\mathcal{A}//E(0))_*\otimes E(\sigma \bar{\xi}_1)\otimes P(\sigma \tau_1)\otimes HH_*(S/p_*(E_0\mathbb{J})). \]
The coaction on elements in $(\mathcal{A}//E(0))_*$ are determined by the coproduct in $\mathcal{A}_*$ and the coaction on elements in $E(\sigma \bar{\xi}_1)\otimes P(\sigma \tau_1)$ is determined by the formula \eqref{coaction sigma}.  
The coaction on a cycle $x_0\otimes x_1 \otimes  \ldots  \otimes x_m$ or $N(x_1\otimes \ldots \otimes x_m)$ for $x_i\in \{ \alpha_1 v_1^{k-1}, v_1^k | k\ge 1\}$ for all $0\le i\le m$ is given by the formulas: \\
\[ \psi(x_0\otimes x_1 \otimes  \ldots \otimes  x_m)= \sum_{i\in I} \bar{\tau}_0\otimes (x_0\otimes \ldots \otimes x_{i-1} \otimes \alpha_1v_1^{k_i-1}\otimes x_{i+1}\otimes \ldots \otimes x_m) + 1\otimes (x_0\otimes x_1 \otimes  \ldots \otimes  x_m) \]
where $x_i=v_1^{k_i}$ for $i\in I\subset \{0,1,\ldots,m\}$ and $x_j=\alpha_1v_1^{k_j}$ if $j\not\in I$ and $0\le j\le m$, and 
\[ \psi(N(x_1 \otimes  x_2\otimes \ldots  \otimes x_m)= \sum_{i\in I} \bar{\tau}_0\otimes N(x_1\otimes \ldots \otimes x_{i-1} \otimes \alpha_1v_1^{k_i-1}\otimes x_{i+1}\otimes \ldots \otimes x_m) + 1\otimes N(x_1 \otimes x_2\otimes \ldots  \otimes x_m) \]
where $x_i=v_1^{k_i}$ for $i\in I$ and $x_j=\alpha_1v_1^{k_j}$ if $j\not\in I$ and $0\le j\le m$. The elements $\alpha_1v_1^{k_0}\otimes \alpha_1v_1^{k_1} \otimes \ldots \otimes \alpha_1v_1^{k_m}$ and $N(\alpha_1v_1^{k_1}\otimes\alpha_1v_1^{k_1} \otimes \ldots \otimes \alpha_1v_1^{k_m})$ are comodule primitives. 
\end{prop}
\begin{proof}
As discussed above, there is an isomorphism 
\[ {H\mathbb{F}_p}_*THH(E_0^*\mathbb{J}) \cong {H\mathbb{F}_p}_*THH(H\mathbb{Z}_{(p)})\otimes {H\mathbb{F}_p}_*THH(\overline{E_0^*\mathbb{J}}).\]
The computation
of ${H\mathbb{F}_p}_*THH(H\mathbb{Z}_{(p)})\cong (\mathcal{A}//E(0))_*\otimes E(\sigma \bar{\xi}_1)\otimes P(\sigma \tau_1)$ along with its $\mathcal{A}_*$-coaction 
is due to B\"okstedt \cite{bok}. The B\"okstedt spectral sequence 
\[ HH_*  ({H\mathbb{F}_p}_*(\overline{E_0^*\mathbb{J}})) \Rightarrow {H\mathbb{F}_p}_*THH(\overline{E_0^*\mathbb{J}}) \]
has input $HH_*(S/p_*(E_0^*\mathbb{J}))$ since 
\begin{align}\label{E0bar} 
{H\mathbb{F}_p}_*(\overline{E_0^*\mathbb{J}})\cong \pi_*(S/p\wedge H\mathbb{Z}_{(p)}\wedge  (S\vee \bigvee_{k\ge 1} \Sigma^{(2p-2)k-1}S/p^{\nu_p(k)+1})))\cong S/p_*(E_0^*\mathbb{J}).
\end{align}
Also, we may regard the B\"okstedt spectral sequence as a tri-graded spectral sequence since the differentials must respect the May filtration grading by Lemma \ref{Bok-May}.
There are no possible differentials in the B\"okstedt spectral sequence that preserve the May filtration and consequently the spectral sequence collapses. 

To compute the coaction on elements in $H_*(THH(\overline{E_0^*\mathbb{J}}))$ we note that $\alpha_1v_1^k$ is by definition the image the fundamental class in the composite 
\[ H_*( \Sigma^{(2p-2)k-1}S) \to H_*(\Sigma^{(2p-2)k-1}S/p^{\nu_p(k)+1}) \to H_*(\overline{E_0^*\mathbb{J}}) \to
H_*(THH(\overline{E_0^*\mathbb{J}}))
\]
and $v_1^k$ is the image of the element $\tau_0$ in 
\[H_*(\Sigma^{(2p-2)k-1}S/p^{\nu_p(k)+1})\cong \Sigma^{(2p-2)k-1}E(\tau_0)\] 
under the composite 
\[ H_*(\Sigma^{(2p-2)k-1}S/p^{\nu_p(k)+1})\to H_*(\overline{E_0^*\mathbb{J}})\to H_*(THH(\overline{E_0^*\mathbb{J}})).\]
Consequently, the Bockstein on $\alpha_1v_1^k$ is $v_1^{k+1}$ for $k\ge0$. This demonstrates that $\alpha_1v_1^k$ is primitive for $k\ge 0$ and the coaction on 
$v_1^{k+1}$ is $\psi(v_1^{k+1})=\tau_0\otimes \alpha_1v_1^k + 1\otimes v_1^{k+1}$ for $k\ge 0$. 

To compute the remaining coactions, first consider a general element of the form 
\[N(\alpha_1^{\epsilon_1}v_1^{i_1}\otimes \ldots \otimes \alpha_1^{\epsilon_m}v_1^{i_m}) \in H_*(THH(\bar{E}_0^*\mathbb{J}))\] 
where $\epsilon_j\in\{0,1\}$ for $1\le j\le m$ and $i_j\ge 1$ for $0\le j\le m$. It is represented by a map
\begin{equation}\label{map N} \Sigma^m \bigwedge_{i=1}^{m}\Sigma^{(2p-2)(\epsilon_i+k_i)-1}S/p^{\nu_p(k_i)+1})\to THH(\bar{E}_0^*\mathbb{J}) \end{equation}
where we let $\Sigma^{(2p-2)(\epsilon_i+k_i)-1}S/p^{\nu_p(k_i)+1}=\Sigma^{(2p-2)(\epsilon_i+k_i)-1}S$ when $\epsilon_i=1$. The map \eqref{map N} is constructed by a simplicial map, which on $m$-simplices is induced by 
\[ S\wedge  \bigwedge_{i=1}^{m}\Sigma^{(2p-2)(\epsilon_i+k_i)-1}S/p^{\nu_p(k_i)+1}) \to S\wedge (\bar{E}_0^*\mathbb{J})^{\wedge m} \to  (\bar{E}_0^*\mathbb{J})^{\wedge m+1} \]
where the simplicial object in the source is the $m$-th simplicial suspension of the constant simplicial spectrum and the object on the right is the cyclic bar complex for $\bar{E}_0^*\mathbb{J}$, whose realization is $THH(\bar{E}_0^*\mathbb{J})$. 
\end{proof}

We now use the $H\mathbb{F}_p$ topological Hochschild-May  spectral sequence in a case where the output is known due to \cite[Thm. 7.15]{MR2171809} in order to detect differentials in the $V(1)$ topological Hochschild-May  spectral sequence.
 
\begin{prop} \label{prop d_0}
The $d^1$-differentials in the $H\mathbb{F}_p$ topological Hochschild-May  spectral sequence exactly correspond to the differentials in the Hochschild-May spectral sequence \eqref{HHM}. Therefore, there is an isomorphism of graded $\mathbb{F}_p$-vector spaces
\[ E_{*,*}^2\cong P(v_1)\otimes (A//E(0))_*\otimes E(\sigma \bar{\xi}_1)\otimes P(\sigma \bar{\tau}_2) \otimes E(\alpha_1)\otimes E(\sigma v_1)\otimes \Gamma(\sigma \alpha_1) \]
\end{prop} 

\begin{proof}
Recall that 
\[ E_{*,*}^1 \cong  (A//E(0))_*\otimes E(\sigma \bar{\xi}_1)\otimes P(\sigma \bar{\tau}_2) \otimes HH_*(S/p_*(E_0^*\mathbb{J})) \]
where all elements in $(A//E(0))_*\otimes E(\sigma \bar{\xi}_1)\otimes P(\sigma \bar{\tau}_2)$ are in May filtration zero. If $x_0\otimes x_1\otimes \ldots x_m$ is a word with letters $x_i\in \{\alpha_1v_1^{k-1},v_1^k | k\ge 1\}$, then 
\begin{align}\label{eq: mfilt} \mfilt(x_0\otimes x_1\otimes \ldots x_m)=\mfilt (x_0) + \mfilt(x_1)+\ldots + \mfilt(x_m) \end{align}
where 
$\mfilt(\alpha_1v_1^{k-1})=\mfilt(v_1^k)=(2p-2)k-1$. 
We know that the abutment is 
\[H_*(j)\otimes E(\lambda_1',\lambda_2)\otimes P(\mu_2)\otimes \Gamma(\sigma b). \] 
To achieve this there must be differentials killing off the classes 
\begin{equation}\label{set} \{ N(\underset{m}{\underbrace{\alpha_1^{\epsilon_1}v_1^{i_1}\otimes \ldots \otimes \alpha_1^{\epsilon_m}v_1^{i_m}}}), \alpha_1^{\epsilon_0}v_1^{i_0}\otimes \underset{m}{\underbrace{\alpha_1^{\epsilon_1}v_1^{i_1}\otimes \ldots \otimes\alpha_1^{\epsilon_m}v_1^{i_m}}} | m\ge 1 \} \end{equation}
when $\sum_{k=1}^{m}i_k+\sum_{k=1}^m\epsilon_1>m$ and $\sum_{k=1}^{m}i_k>1$ for any $\epsilon_0\in \{0,1\}$ and $i_0\ge 0$. The elements in \eqref{set} come in pairs with one element in May filtration one lower and topological degree one more than the other. All the other generators must survive to a later page and therefore there is no differential pattern other than the one stated, which would produce the desired abutment. 
\end{proof}
\begin{rem}
We could also produce the differentials in Proposition \ref{prop d_0} by explicitly representing an element, such as $1\otimes v_1\otimes v_1\otimes v_1=N(v_1\otimes v_1\otimes v_1)$, as a generator of 
\[ H_*(H\mathbb{Z}\wedge \Sigma^{2p-3}H\mathbb{F}_p \wedge \Sigma^{2p-3}H\mathbb{F}_p\wedge \Sigma^{2p-3}H\mathbb{F}_p)\subset H_*(E_0^*\mathbb{J}^{\wedge 4})\] 
For example, the differential 
\[ d_1(1\otimes v_1\otimes v_1)=1\otimes v_1^2\otimes v_1 -1\otimes v_1\otimes v_1^2=N(v_1^2\otimes v_1)\] 
arises from the fact that 
$(E_0^*\mathbb{J})^{\wedge 4}\overset{\id \wedge \mu }{\longrightarrow} (E_0^*\mathbb{J})^{\wedge 3}$ 
is the zero map in degrees greater than zero, but  this is an artifact of the fact that the products of elements in positive degree such as $\alpha_1\cdot \alpha_1$ are zero in $\pi_*j$. 
However, in $\pi_*S/p\wedge j$ there is a relation $\beta (\alpha_1)\cdot \beta (\alpha_1)=v_1\cdot v_1=v_1^2=\beta(\alpha_2)$, which is not visible in $S/p_*(E_0^*j)$ because $v_1^2=0$. We claim that this multiplicative relation is reintroduced by differentials like $d_1(N(v_1\otimes v_1 \otimes v_1))=N(v_1^2\otimes v_1)$. The element $N(v_1^2\otimes v_1)$ can be thought of as $\sigma v_1^2 \cdot \sigma v_1$ using the shuffle product and the notation $1\otimes v_1=\sigma v_1$. If $v_1^2$ were a square of an element in the input, then since $\sigma$ is a derivation it would be the case that $\sigma v_1^2 \cdot \sigma v_1=0$. Since the fact that it is not a square any longer is an artifact of the associated graded construction, this needs to be corrected by differentials in the spectral sequence and this differential does just that. 

In other words, there is a nontrivial map 
\[
\xymatrix{
H\mathbb{Z}\wedge \Sigma^{2p-3}H\mathbb{F}_p \wedge \Sigma^{2p-3}H\mathbb{F}_p\wedge \Sigma^{2p-3}H\mathbb{F}_p \ar[d] \\
 \left ( H\mathbb{Z}\wedge \Sigma^{4p-5}H\mathbb{F}_p\wedge \Sigma^{2p-3}H\mathbb{F}_p\right ) \vee \left ( H\mathbb{Z}\wedge \Sigma^{2p-3}H\mathbb{F}_p\wedge \Sigma^{4p-5}H\mathbb{F}_p\right ),
 }
 \]
that shifts May filtration by one, representing the differential. This sort of argument is standard for the spectral sequence of a double complex and our claim is that it also holds in the spectral setting. We also claim that the differentials in Proposition \ref{prop d_0} can all be constructed using a similar argument, although it would be quite tedious to compute them this way and therefore we give the simpler argument. 
\end{rem}
\begin{prop} \label{prop d_1} 
There are differentials 
\begin{align*}
d^{2p-3}(\bar{\xi}_1)\dot{=}&\alpha_1, & d^{2p-3}(\bar{\tau}_1)\dot{=}&v_1, & d^{2p-3}(\sigma \bar{\tau}_1 )\dot{=} &\sigma v_1,\\
d^{2p-3}(\sigma \bar{\xi}_1 )=&\sigma \alpha_1,  & d^{2p-2}(\tilde{\tau}_1v_1^k)\dot{=}&v_1^{k+1}, & d^{2p-2}(\bar{\xi}_1v_1^k)\dot{=}&\alpha_1v_1^k, \\
d^{2p-2}(\tilde{\tau}_1\alpha_1v_1^{k-1})=&\alpha_1v_1^k & &&& 
\end{align*}
for $k\ge 1$
in the $H\mathbb{F}_p$ topological Hochschild-May  spectral sequence 
\[ (H\mathbb{F}_p)_{s,t}(THH(E_0^*\mathbb{J}))\Rightarrow (H\mathbb{F}_p)_s(THH(j)) \]
and no further differentials besides those generated using the Leibniz rule from the differentials above. 
The surviving classes 
\[ \{ \bar{\xi}_1^{p-1}\alpha_1, \sigma \bar{\xi}_1 \gamma_{p-1}\sigma \alpha_1, \gamma_p(\sigma \alpha_1), (\sigma \bar{\tau}_1)^p, (\sigma \bar{\tau}_1)^{p-1}\sigma v_1 \}\]
map to the classes 
$ \{b, \sigma \tilde{\xi}_1^p, \sigma b, \sigma \tilde{\tau}_2, \sigma \tilde{\xi}_2 \}$ 
respectively in ${H\mathbb{F}_p}_*(THH(j))$ and there no hidden multiplicative extensions in the $H\mathbb{F}_p$ topological Hochschild-May spectral sequence. All other surviving classes map to classes of the same name. 
\end{prop} 
\begin{proof} 
There is a map of $H\mathbb{F}_p$ topological Hochschild-May  spectral sequences
\begin{align}\label{map of homological thh-may ss}
\xymatrix{ 
(H\mathbb{F}_p)_{*,*}(E_0^*\mathbb{J}) \ar[d] \ar@{=>}[r] & (H\mathbb{F}_p)_*(j) \ar[d] \\
(H\mathbb{F}_p)_{*,*}(THH(E_0^*\mathbb{J})) \ar@{=>}[r] & (H\mathbb{F}_p)_*(THH(j) )
}
\end{align}
induced by the map $*\to S^1_{\bullet}$  of simplicial finite sets, where the elements 
$\bar{\xi}_1,\alpha_1v_1^k,\bar{\tau}_1,$ and $v_1^k$ all map to elements of the same name. 
This produces the differentials $d^{2p-3}(\bar{\xi}_1)\dot{=}\alpha_1,$ $d^{2p-3}(\bar{\tau}_1)\dot{=}v_1,$ $d^{2p-2}(\tilde{\tau}_1v_1^k)\dot{=}v_1^{k+1},$ 
and $d^{2p-2}(\bar{\xi}_1v_1^k)\dot{=}\alpha_1v_1^k.$ 

We also know that the abutment $(H\mathbb{F}_p)_m(THH(j))$ is trivial in the range $0< m <2p^2-1$, which forces the differentials $d^{2p-3}(\sigma \bar{\tau}_1 )\dot{=} \sigma v_1$ and 
$d^{2p-3}(\sigma \bar{\xi}_1 )=\sigma \alpha_1$. The resulting $E^{2p-1}$-page is additively isomorphic to ${H\mathbb{F}_p}_*(THH(j))$ with the specified correspondence in the proposition and therefore the spectral sequence must collapse at this page.  In Figure \ref{sseq:HFptHM}, we draw the $E^{2p-1}\cong E^{\infty}$-page, where write $b=\alpha_1\bar{\xi}_1^{p-1}$ and $\gamma_k^{\prime}=\gamma_k(\gamma_p(\sigma \alpha_1))$, $\sigma \bar{\xi}_1^p= \sigma \bar{\xi} \gamma_{p-1}\sigma \alpha_1$ and $\sigma \bar{\xi}_2=(\sigma \bar{\tau}_1)^{p-1}\sigma v_1$ for brevity. 

\begin{figure} 
\begin{center}
\tiny
\[ 
\begin{sseq}[entrysize=.7cm,ylabelstep=1,ylabels={0;;;2p-3;;;2p^2-5p+3;;2p^2-3p;;2p^3-p-3;;;4p^2-8p+3;;;4p^2-6p;;},xlabelstep=1,xlabels={0;;;;2p^2-2p;;;;2p^2-2;;;;4p^2-4p;;;;4p^2-2p-2;;;;;;}]{19}{17}
\ssdrop{1}
\ssmove{3}{3}
\ssdrop{b}
\ssmove{1}{-3}
\ssdrop{\bar{\xi}_1^p}
\ssmove{0}{8}
\ssdrop{\gamma_1^{\prime}}
\ssmove{1}{-2}
\ssdrop{\sigma \bar{\xi}_1^p}
\ssmove{3}{-6}
\ssdrop{\bar{\xi}_2}
\ssmove{1}{0}
\ssdrop{\bar{\tau}_2}
\ssmove{0}{3}
\ssdrop{\sigma \bar{\xi}_2}
\ssmove{2}{0}
\ssdrop{b\bar{\xi}_1^p}
\ssmove{0}{7}
\ssdrop{\gamma_1^{\prime} b}
\ssmove{1}{-10}
\ssdrop{\bar{\xi}_1^{2p}}
\ssmove{0}{8}
\ssdrop{\gamma_1^{\prime}\bar{\xi}_1^p}
\ssdrop{b\sigma \bar{\xi}_1^p}
\ssmove{0}{8}
\ssdrop{\gamma_2^{\prime}}
\ssmove{1}{-3}
\ssdrop{\gamma_1^{\prime}\sigma \bar{\xi}_1^p}
\ssmove{2}{-10}
\ssdrop{b\bar{\xi}_2}
\ssmove{1}{-3}
\ssdrop{\bar{\xi}_1^p\bar{\xi}_2}
\ssmove{0}{3}
\ssdrop{b\bar{\tau}_2}
\ssmove{0}{3}
\ssdrop{b\sigma \bar{\xi}_2}
\ssmove{1}{-6}
\ssdrop{\bar{\xi}_1^p\bar{\tau}_2}
\ssmove{0}{10}
\ssdrop{\gamma_1^{\prime}\sigma \bar{\xi}_2}
\ssmove{1}{-2}
\ssdrop{\sigma \bar{\xi}_1^p\sigma \bar{\xi}_2}
\end{sseq}
\]
\end{center}
\caption{The $E^{\infty}_{s,t}$-page of the $H\mathbb{F}_p$ topological Hochschild-May spectral sequence for $p\ge 3$ for $s\le 4p^2-2p$ and all $t$.}\label{sseq:HFptHM} 
\end{figure}

We may describe the multiplicative filtration $G_{\bullet}^{\prime}$ of the abutment $(H\mathbb{F}_p)_*(THH(j)))$ whose associated graded is the $E^{\infty}$-page as follows. 
First, of course, $G_0^{\prime}=(H\mathbb{F}_p)_*(THH(j)))$. We define $G_{s}^{\prime}$  for $s\ge 1$ by
\[ G_s^{\prime}= \{ x\in G_0^{\prime} : \mfilt{x}\ge s \}.\]
By inspection, the associated graded of the filtration $G_{\bullet}^{\prime}$ is exactly 
\[ \bigoplus_{i\ge 0} G_{i}^{\prime}/G_{i+1}^{\prime}= E^{\infty}_{*,*}.\]
Note that there were hidden multiplicative extensions determined in Remark \eqref{HM remark}, however the multiplicative extensions only involved classes which do not survive the $H\mathbb{F}_p$ topological Hochschild-May spectral sequence, so in fact they do not appear here. We also determined that there are no multiplicative extensions in the top spectral sequence of \eqref{map of homological thh-may ss} in Lemma \ref{lem: homology of ag} and in fact this map of spectral sequences splits off of the bottom spectral sequence. The remaining elements have the property that 
\[ \text{mfilt}(xy)=\text{mfilt}(x)+\text{mfilt}(y)\]
by \eqref{eq: mfilt} so the products are nontrivial already at the $E^{\infty}$-page. We observe that there is an isomorphism of graded $\mathbb{F}_p$-algebras
\[E^{\infty}_{*,*} \cong (H\mathbb{F}_p)_*(THH(j)).\]
This is visible, in a range, by Figure \ref{sseq:HFptHM}.  
Therefore, there cannot be multiplicative extensions because then the abutment would not have the correct multiplicative structure, contradicting the fact that the $H\mathbb{F}_p$ topological Hochschild-May spectral sequence is a strongly convergent multiplicative spectral sequence. 
\end{proof} 
\begin{rem} The behavior of the differentials above leads us to speculate that the differentials in the topological Hochschild-May  spectral sequence commute with the operation $\sigma$. We plan to return to this in future work. 
\end{rem} 
\begin{prop} \label{computation}There is an isomorphism 
 \[ V(1)_*(THH(E_0^*\mathbb{J}))\cong  E(\lambda_1, \epsilon_1)\otimes P(\mu_1)\otimes HH_*(S/p_*E_0^*\mathbb{J}) \] 
 where $|\epsilon_1|=|\lambda_1|=|\sigma \tilde{v}_1|=2p-1$, $|\alpha_1|=2p-3$, $|\mu_1|=2p$, $|\tilde{v}_1|=2p-2$, and $|\sigma \alpha_1|=2p-2$. 
 \end{prop} 
 \begin{proof}
 We can compute ${H\mathbb{F}_p}_*(V(1)\wedge THH(j))$ where the input is ${H\mathbb{F}_p}_*(V(1)\wedge THH(E_0^*\mathbb{J}))$, using the $H\mathbb{F}_p\wedge V(1)$ topological Hochschild-May  spectral sequence. The differentials are the same and the classes $\bar{\tau}_0$ and $\bar{\tau}_1$ map to classes of the same name in the output. This is useful because there is a map of spectral sequences from the $V(1)$ topological Hochschild-May spectral sequence to the $H\mathbb{F}_p\wedge V(1)$ topological Hochschild-May  spectral sequence induced by the map of $S$-algebras
\[\xymatrix{ S\wedge V(1) \ar[rr]^{\eta \wedge \id_{V(1)}} && H\mathbb{F}_p\wedge V(1) }\] 
where $\eta\co S\rightarrow H\mathbb{F}_p$ is the unit map of $H\mathbb{F}_p$. 
 Due to Lemma \ref{prim}, the map 
 \[ V(1)_*(THH(E_0^* \mathbb{J} ))\lra (H\mathbb{F}_p\wedge V(1))_*(THH(E_0^* \mathbb{J}))\]
  includes $V(1)_*(THH(E_0^*\mathbb{J}))$ into $(H\mathbb{F}_p\wedge V(1))_*(THH(E_0^*\mathbb{J}))$ as the $\mathcal{A}_*$-comodule primitives. 
  
  By Lemma \ref{Bok1}, the elements 
 \begin{equation}\label{comod prim} \{ \alpha_1, v_1-\bar{\tau}_0\alpha_1 , \sigma \tilde{\tau}_1-\bar{\tau}_0\sigma \bar{\xi}_1,  \sigma \bar{\xi}_1, \sigma \alpha_1, \sigma v_1-\bar{\tau}_0\sigma\alpha_1, \tilde{\tau}_1-\bar{\tau}_1\} \end{equation} 
are comodule primitives where we write $\tilde{\tau}_1$ to distinguish the class in ${H\mathbb{F}_p }_*(THH(E_0^*\mathbb{J}))$ from the class $\bar{\tau}_1\in {H\mathbb{F}_p}_*(V(1))$. We rename these classes respectively 
 \[ \{ \alpha_1, v_1, \mu_1, \lambda_1, \sigma\alpha_1, \sigma v_1, \epsilon_1\}. \] 
 In addition, there is a corresponding comodule primitive for each algebra generator of $HH_*(S/p_*\mathbb{J})$ which can be computed by simply subtracting the necessary terms as we did to produce the comodule primitive $v_1-\bar{\tau}_0\alpha_1$. 
 Therefore, the subalgebra of comodule primitives is isomorphic to 
 \[ E(\lambda_1,\epsilon_1 )\otimes P(\mu_1)\otimes HH_*(S/p_*(E_0^*\mathbb{J}))\] 
 and the result follows from Lemma \ref{Bok1} and Lemma \ref{prim}. 
\end{proof}
 We now consider the map of topological Hochschild-May  spectral sequences 
 \[ 
 \xymatrix{ 
 V(1)_*(THH(E_0^*\mathbb{J}))\ar[d]^{f} \ar@{=>}[r] & V(1)_*(THH(j))\ar[d]  \\
 (H\mathbb{F}_p\wedge V(1))_*(THH(E_0^*\mathbb{J})) \ar@{=>}[r] &  (H\mathbb{F}_p\wedge V(1))_*(THH(j)) }\]
 induced by the map 
 \[ \eta \wedge \id_{V(1)}\co S\wedge V(1)\lra H\mathbb{F}_p\wedge V(1) \]
 where $\eta\co S\rightarrow H\mathbb{F}_p$ is the unit map of the ring spectrum $H\mathbb{F}_p$. 
 \begin{prop}\label{prop d_0 2} 
 The only $d^1$ differentials in the $V(1)$ topological Hochschild-May  spectral sequence are those in the algebraic Hochschild-May spectral sequence, consequently, there is an additive isomorphism 
 \[ E_{*,*}^2\cong E( \lambda_1,\epsilon_1) \otimes P(\mu_1) \otimes E(\alpha_1 )\otimes P(v_1) \otimes \Gamma(\sigma \alpha_1)\otimes E(\sigma v_1)  \]
 where $|\lambda_1|=(2p-1,)$, $|\epsilon_1|=(2p-1,0)$, $|\mu_1|=(2p,0)$, $|\alpha_1v_1^{k-1}|=(2p-3,(2p-2)k-1)$, $|v_1^k|=((2p-2)k,(2p-2)k-1)$ for $k\ge 1$, $|\gamma_{p^k}(\sigma \alpha_1)|=((2p-2)p^k,(2p-3)k)$ and $|\sigma v_1|=(2p-1,2p-2)$. Multiplicatively, there is an isomorphism 
\[ E_{*,*}^2\cong E( \lambda_1,\epsilon_1) \otimes P(\mu_1) \otimes \mathbb{F}_p[x_i | i\ge 1]/\mathfrak{m} \otimes \Gamma(\sigma \alpha_1)\otimes E(\sigma v_1) \]
where $x_{2i}$ corresponds to $v_1^i$ and $x_{2i-1}$ corresponds to $\alpha_1v_1^i$ for $i\ge 1$ and $\mathfrak{m}=(x_1,x_2,\ldots )$.
 \end{prop}
 \begin{proof}
This is an easy consequence of Proposition \ref{prop d_0} since the map 
\[V(1)\wedge THH(E_0^*\mathbb{J})\to {H\mathbb{F}_p}_*(V(1)\wedge THH(E_0^*\mathbb{J}))\] 
is an an isomorphism onto the  sub-algebra of co-module primitives, and in particular it is injective. The $d^1$-differentials therefore correspond to the $d^1$-differentials in ${H\mathbb{F}_p}_*(V(1)\wedge THH(E_0^*\mathbb{J}))$
\end{proof}
 \begin{prop} \label{prop d1}The only $d^{2p-3}$ differentials in the $V(1)$ topological Hochschild-May  spectral sequence are  
\begin{align*}
d^{2p-3}(\lambda_1)\dot{=}&\sigma \alpha_1, \\
d^{2p-3}(\epsilon_1)\dot{=} &v_1, \\
d^{2p-3}(\mu_1)\dot{=}& \sigma  v_1, 
\end{align*}
and those differentials generated by these differentials under the Leibniz rule. There are $d^{2p-2}$ differentials 
\begin{align*}
d^{2p-2}(\epsilon_1v_1^k)\dot{=} &v_1^{k+1}, \text{ and }\\
d^{2p-3}(\epsilon_1\cdot \alpha_1v_1^{k-1})\dot{=}&\alpha_1v_1^k 
\end{align*}
for $k\ge 1$.
The $E^{2p-1}$-page of the $V(1)$ topological Hochschild-May  spectral sequence is therefore isomorphic to 
\[ E(\alpha_1, \lambda_1^{\prime}, \lambda_2 ) \otimes P(\mu_2)\otimes \Gamma(\sigma b) \] 
where $\lambda_1^{\prime}=\lambda_1\gamma_{p-1}(\sigma \alpha_1)$, $\gamma_p(\sigma \alpha_1)=\sigma b$, $\lambda_2=\sigma v_1 \mu_1^{p-1}$, and $\mu_2=\mu_1^p$. The bidegrees of these elements are therefore
$|\lambda_1'|=(2p^2-2p+1,2p^2-5p+4)$, $|\sigma b|=(2p^2-2p,2p^2-3p)$, $|\lambda_2|=(2p^2-1,2p-1)$, and $|\mu_2|=(2p^2,0)$. 
\end{prop}
\begin{proof}
The classes 
 $ \{ v_1, \mu_1, \lambda_1, \sigma\alpha_1,  \sigma v_1, \epsilon_1 \} $
 in the $V(1)$ topological Hochschild-May  spectral sequence map to the classes
  \[ \{ v_1-\bar{\tau}_0\alpha_1 , \sigma \tilde{\tau}_1-\bar{\tau}_0\bar{\xi}_1,  \sigma \bar{\xi}_1, \sigma \alpha_1, \sigma v_1-\bar{\tau}_0\sigma\alpha_1, \tilde{\tau}_1-\bar{\tau}_1\} \] 
in the $H\mathbb{F}_p\wedge V(1)$ topological Hochschild-May  spectral sequence under the map of spectral sequences, denoted $f$, induced by the map $S\wedge V(1)\to H\mathbb{F}_p\wedge V(1)$. There are trivial differentials in the $H\mathbb{F}_p\wedge V(1)$ topological Hochschild-May  spectral sequence
\[ d_r(\bar{\tau}_0)=d_r(\bar{\tau}_1)=0 \] 
for $r\ge 1$ 
and nontrivial differentials 
\begin{align*}
d^{2p-3}(\bar{\xi}_1)\dot{=}&\alpha_1,&  d^{2p-3}(\sigma \bar{\xi} )\dot{=}&\sigma \alpha_1, \\
d^{2p-3}(\tilde{\tau}_1)\dot{=}&v_1, & d^{2p-3}(\sigma \bar{\tau}_1 )\dot{=}& \sigma v_1, \\
d^{2p-2}(\tilde{\tau}_1v_1^k)\dot{=}& v_1^{k+1} &d^{2p-2}(\bar{\xi}_1v_1^k)\dot{=}&\alpha_1v_1^{k}  \\
d^{2p-2}(\tilde{\tau}_1\alpha_1v_1^{k-1})\dot{=}& \alpha_1v_1^{k} && 
\end{align*}
for $k\ge 1$ as well as the $d^1$ differentials in the Hochschild-May spectral sequence \eqref{HHM} by Propositions \ref{prop d_1} and \ref{computation}. 
We will use the formula $fd^r=d^rf$ to compute the differentials. 
Notice that the map $f$ is still injective on the $E^2$-page of the spectral sequences so it makes sense to use the formula $d^r(x)=f^{-1}d^r(f(x))$. We therefore produce differentials
\begin{align*}
d^{2p-3}(\lambda_1)=& f^{-1}(d^{2p-3}(\sigma \bar{\xi}_1))= f^{-1}(\sigma \alpha_1)=\sigma \alpha_1, \\
d^{2p-3}(\epsilon_1)= &f^{-1}d^{2p-3}(\tilde{\tau}_1-\bar{\tau}_1)=f^{-1}(v_1)=v_1 \\
d^{2p-3}(\mu_1)= &f^{-1}(d^{2p-3}(\sigma \tilde{\tau}_1-\bar{\tau}_0\bar{\xi}_1))=f^{-1}(\sigma(v_1)-\bar{\tau}_0\alpha_1)=\sigma(v_1) 
\end{align*}
in the $V(1)$ topological Hochschild-May  spectral sequence as desired. By Figure \eqref{sseq:2}, there are no other possible $d^r$ differentials when $2\le r\le 2p-2$ for bidegree reasons in columns to the left of $4p^2$. The only exceptions to this are the possible differential on $\lambda_1^{\prime}$ hitting $\sigma b$ and the possible differential on $\mu_2$ hitting $\lambda_2$, but these can both be ruled out. Recall that $\lambda_1^{\prime}=\lambda_1\gamma_{p-1}(\sigma \alpha_1)$ and $\gamma_p (\sigma \alpha_1)=\sigma b$. There is a differential $d_{2p-3}(\lambda_1)=\sigma \alpha_1$ and therefore $d_{2p-3}(\lambda_1\gamma_{p-1}(\sigma\alpha_1)=\gamma_{p-1}(\sigma \alpha_1)\cdot \sigma \alpha_1=p\gamma_p(\sigma \alpha_1)=0$ modulo $p$. Similarly, $\mu_2=\mu_1^p$ and $\lambda_2=\sigma v_1 \mu_1^{p-1}$ and the differential $d_{2p-3}(\mu_1)=\sigma v_1$ implies that $d_{2p-3}(\mu_1^p)=p\sigma v_1 \mu_1^{p-1}=0$. Therefore, since all of the algebra generators are in range the depicted in Figure \eqref{sseq:2} except for the elements $\gamma_{p^k}(\sigma \alpha_1)$ for $k\ge 1$, the only remaining possible differentials are differentials whose source is $\gamma_{p^k}(\sigma \alpha_1)$ for $k\ge 1$. We will eliminate this possibility by directly analyzing which elements are in topological degree $(2p^2-2p)p^k-1$, one column to the left of $\gamma_{p^k}(\sigma b)$, and in strictly higher May filtration. Consider a general monomial 
\begin{equation}\label{monomial} \lambda_1^{e_1}\epsilon_1^{e_2}\alpha_1^{e_3}\sigma v_1^{e_4}\mu_1^{\ell}v_1^j\gamma_m(\sigma \alpha_1)\end{equation}
whose topological degree is 
\[ (2p-2)e_1+(2p-1)e_2+(2p-3)e_3+(2p-1)e_4+(2p)\ell+(2p-2)j+(2p-2)m\]
and whose May filtration is $(2p-3)e_4+(2p-2)(j+e_3)-1+(2p-3)m$ where $e_1,e_2,e_3,e_4\in\{0,1\}$, $\ell,j,m\ge 0$, and $j+e_3>0$. When $e_1,e_2,e_3,e_4\in \{0,1\}$, $\ell,j,m\ge 0$, and  $j+e_3=0$, the May filtration of \eqref{monomial}
is $(2p-3)e_4+(2p-3)m$.
In order for this element to be the target of a differential on $\gamma_{p^k}(\sigma \alpha_1)$, the relation
\[ (2p-2)e_1+(2p-1)e_2+(2p-3)e_3+(2p-1)e_4+(2p)\ell+(2p-2)j+(2p-2)m=(2p-2)p^k-1\]
must hold. By rearranging terms, we see that  the relation
\[ (2p-2)(e_1+e_2+e_3+e_4+\ell+j+m-p^k)+e_2-e_3+e_4+2\ell+1=0\]
must hold. Since $e_2-e_3+e_4+2\ell+1\ge 0$, the relation will only hold if 
\begin{equation}\label{relation 1}e_1+e_2+e_3+e_4+\ell+j+m-p^k=0 \end{equation} and  
\begin{equation}\label{relation 2} e_2-e_3+e_4+2\ell+1=0 .\end{equation} 
However, we know the May filtration of \eqref{monomial} must be greater than the May filtration of $\gamma_{p^k}(\sigma \alpha_1)$ and therefore when $j+e_3>0$, 
\[ (2p-3)e_4+(2p-2)(j+e_3)-1+(2p-3)m>(2p-3)p^k \]
and when $j+e_3=0$
\[ (2p-3)e_4+(2p-3)m> (2p-3)p^k.\]
When $j+e_3=0$, the relation reduces to 
\[ e_4+m-p^k>0 \]
and consequently 
$e_1+e_2+e_3+e_4+\ell+j+m-p^k>e_1+e_2+e_3+\ell+j\ge 0$
so \eqref{relation 1} does not hold and therefore the corresponding element \eqref{monomial} cannot be the target of a differential on $\gamma_{p^k}(\sigma b)$. When $j+e_3=0$, then 
\[ (2p-3)e_4+(2p-2)(j+e_3)-1+(2p-3)m>(2p-3)p^k \]
reduces to 
\[ (2p-3)(e_4+j+e_3+m-p^k) >1-j-e_3\]
so $e_4+j+e_3+m-p^k >\frac{1-j-e_3}{2p-3}$ and therefore 
$e_1+e_2+e_3+e_4+\ell+j+m-p^k>e_1+e_2+\ell+\frac{1-j-e_3}{2p-3}\ge 0$
when $j<2p-3$. Thus, we have reduced to the case when $j>2p-3$. We now note that $j$ is the exponent of $v_1$ in \eqref{monomial}. Consider the map of spectral sequences induced by the Hurewicz map 
\begin{equation}\label{hur} V(1)\to H\mathbb{F}_p\wedge V(1) .\end{equation}
There is a differential in the target spectral sequence 
\[d_{2p-2-\epsilon}(\lambda_1^{e_1}\tau_1\alpha_1^{e_3}\sigma v_1^{e_4-1}\mu_1^{\ell}v_1^j\gamma_m(\sigma \alpha_1))=\lambda_1^{e_1}\alpha_1^{e_3}\sigma v_1^{e_4}\mu_1^{\ell}v_1^j\gamma_m(\sigma \alpha_1)\]
where $\epsilon=1$ if $e_1=e_3=e_4-1=\ell=j=m=0$ and $\epsilon=0$ otherwise. Since 
$\lambda_1^{e_1}\epsilon_1\alpha_1^{e_3}\sigma v_1^{e_4-1}\mu_1^{\ell}v_1^j\gamma_m(\sigma \alpha_1)$ maps to 
$\lambda_1^{e_1}\tau_1\alpha_1^{e_3}\sigma v_1^{e_4-1}\mu_1^{\ell}v_1^j\gamma_m(\sigma \alpha_1)$ and $\lambda_1^{e_1}\alpha_1^{e_3}\sigma v_1^{e_4}\mu_1^{\ell}v_1^j\gamma_m(\sigma \alpha_1)$ maps to an element of the same name in the map of $E^2$-pages of spectral sequences induced by the map \ref{hur}, the target of this differential cannot be hit by a differential at an earlier page or else the compatibility of the map of spectral sequences induced by \eqref{hur} with the differentials would be violated. This implies that $e_2=1$ and $j\ge 2p-3$. We now recall that that the relation \eqref{relation 2} also must be satisfied and when $e_2=1$, then $e_2-e_3+e_4+2\ell+1=2-e_3+e_4+2\ell>0$ and thus  \eqref{relation 2} cannot be satisfied. Therefore, there are no possible differentials with $\gamma_{p^k}(\sigma \alpha_1)$ as a source. Note that this same argument works for all differentials of the form $d^r(\gamma_{p^k}(\sigma \alpha_1))$ for $r\ge 1$. We will use this fact in the proof of Theorem \ref{main thm} as well.

Consequently, there are no possible differentials of length $2\le r\le 2p-2$ and in fact there are no other possible differentials besides the ones stated and the ones implied by the Leibniz rule for $r<2p^2-5p+3$ by the same argument.
\end{proof}
\begin{rem}
Note that the only elements in the $E^{2p-1}$ page of the $V(1)$ topological Hochschild-May  spectral sequence that are in the kernel of the map to the $E^{2p-1}$ page of the $H\mathbb{F}_p$ topological Hochschild-May  spectral sequence induced by the Hurewicz map are the multiples of $\alpha_1$. Since the differential on $\alpha_1$ is trivial for bidegree reasons and the spectral sequence obeys the Leibniz rule, the only possible differentials of longer length are differentials with a multiple of $\alpha_1$ as a target. 
\end{rem}
\begin{lem}
There is an isomorphism of graded $\mathbb{F}_p$-algebras
\[ V(1)_*(THH(j;\ell)) \cong E(\lambda_1',\lambda_2)\otimes P(\mu_2)\otimes \Gamma(\sigma b) \] 
\end{lem} 
\begin{proof}
 Note that there are equivalences 
 \[ V(1)\wedge THH(j;\ell)\simeq THH(j;H\mathbb{F}_p)\simeq H\mathbb{F}_p\wedge_j THH(j) \] 
 and that $H\mathbb{F}_p\wedge_j THH(j)$ is an $H\mathbb{F}_p$-algebra. 
 We can therefore apply Lemma \ref{prim} and Theorem \ref{HFpj} to compute $V(1)_*THH_*(j;\ell)$. 
 The result is the algebra of comodule primitives in ${H\mathbb{F}_p}_*THH(j;H\mathbb{F}_p)$. 
 Since there is an equivalence $THH(j;H\mathbb{F}_p)\simeq THH(j)\wedge_j H\mathbb{F}_p$, we can compute this using the Eilenberg-Moore spectral sequence 
 \[ \Tor^{H_*(j)}_{*,*} ( H_*(THH(j)),H_*(H\mathbb{F}_p)) \Rightarrow  H_*(THH(j;H\mathbb{F}_p)).\]
 Since $H_*THH(j)$ is a free $H_*j$-module by Theorem \ref{HFpj}, this spectral sequence collapses producing an isomorphism
\[ {H\mathbb{F}_p}_*(THH(j;H\mathbb{F}_p)) \cong \mathcal{A}_*\otimes E(\sigma \tilde{\xi}_1^p, \sigma \tilde{\xi}_2)\otimes P(\sigma \tilde{\tau}_2)\otimes  \Gamma(\sigma b) ,\] 
where the coaction on elements in $\mathcal{A}_*$ is determined by the coproduct and the coaction on the remaining elements is determined in Theorem \ref{HFpj}. 

The subalgebra of comodule primitives is isomorphic to 
$E(\lambda_1', \lambda_2)\otimes P(\mu_2)\otimes  \Gamma(\sigma b)$
where 
$\lambda_1'=\sigma \tilde{\xi}_1^p -\bar{\tau}_0\sigma b,$
$\lambda_2= \sigma \bar{\xi}_2 - \bar{\xi}_1 \otimes \sigma \tilde{\xi}_1^p - \bar{\tau}_1\otimes \sigma b,$ and 
$ \mu_2 = \sigma \tau_2 - \bar{\tau}_0 \sigma \tilde{\xi}_2 - \bar{\tau}_1 \sigma \tilde{\xi}_1^p +\tau_0\tau_1 \sigma b. $
\end{proof} 
We have another approach to computing $THH_*(j; j/(p,v_1))=V(1)_*(THH(j))$ by filtering the coefficients $j/(p,v_1)$ using the short filtration 
\begin{equation}\label{filt mod} 0 \lra \Sigma^{2p-3}H\mathbb{F}_p \lra j/(p,v_1) \end{equation}
with associated graded $j$-module $H\mathbb{F}_p\ltimes \Sigma^{2p-3}H\mathbb{F}_p$. It is important to note that this is not a spectral sequence of $V(1)\wedge j$-algebras because $V(1)\wedge j$ is not a commutative ring spectrum and therefore the filtration \eqref{filt mod} is not itself a cofibrant decreasingly filtered commutative monoid in spectra in the sense of \cite[Def. 3.2.2]{thhmay}. The filtration of \eqref{filt mod} is just a cofibrant decreasingly filtered symmetric  $\mathbb{J}$-module in the sense of \cite[Def. A.1.1]{thhmay}, which suffices to construct the spectral sequence without multiplicative structure. 
We use the topological Hochschild-May  spectral sequence with filtered coefficients as follows 
\begin{equation} \label{ss filt coeff} THH_{s,t}(j;H\mathbb{F}_p\ltimes \Sigma^{2p-3}H\mathbb{F}_p) \Rightarrow THH_s(j; j/(p,v_1)). \end{equation}
This spectral sequence reduces to the long exact sequence 
\begin{equation}\label{eq les} \xymatrix{ 
\dots \ar[r] &  \pi_{k-2p+3}(THH(j;H\mathbb{F}_p) \ar[r]   & \pi_k(THH(j;j/(p,v_1)) \ar`r_l[ll] `l[dll] [dll] \\
 \pi_k(THH(j;H\mathbb{F}_p)) \ar[r] &  \pi_{k-2p+2}(THH(j;H\mathbb{F}_p)) \ar[r] &   \dots  }
\end{equation}
where two out of three terms are known. We claim that this exact sequence demonstrates that the $V(1)$ topological Hochschild-May  spectral sequence cannot collapse at $E_{2p-2}$. The author owes Eva H\"oning for giving some evidence that there must be a longer differential in personal communication, since the author originally had a faulty argument that said that the differential on $\lambda_2=(\mu_1)^{p-1}\sigma \tilde{v}_1$ was zero. 
\begin{prop} \label{prop dp-1}
There is a differential 
\[ d^{2p^2-5p+3} ((\mu_1)^{p-1}\sigma \tilde{v}_1)\dot{=}\alpha_1 \lambda_1\gamma_{p-1}(\sigma \alpha_1) \] 
in the $V(1)$ topological Hochschild-May  spectral sequence and no remaining differentials. 
\end{prop} 
\begin{figure}
\begin{center}
\tiny
\[ 
\begin{sseq}[entrysize=.35cm,ylabelstep=2,ylabels={0;;2p-3;;\vdots;;2p^2-5p+3;;2p^2-3p;;2p^2-p-3;;\vdots;;4p^2-8p+3;;4p^2-6p;;4p^2-4p-3;;},xlabelstep=18,xlabels={0;;;;;;;;;;;;;;;;;;2p^2;;;;;;;;;;;;;;;;;;4p^2}]{37}{20}
\ssdrop{1}

\ssmove{3}{2}
\ssdrop{\alpha_1}

\ssmove{9}{6}
\ssdrop{\sigma b}
\ssmove{1}{-2}
\ssdrop{\lambda_1^{\prime}}
\ssmove{1}{1}
\ssmove{1}{3}
\ssdrop{\alpha_1 \sigma b }
\ssmove{1}{-2}
\ssdrop{\alpha_1\lambda_1^{\prime}} \ssname{a} 

\ssmove{1}{-6}
\ssdrop{\lambda_2} \ssname{b}
\ssgoto a \ssgoto b \ssstroke
\ssmove {1} {-2} 
\ssdrop{\mu_2}
\ssmove{2}{4}
\ssdrop{\alpha_1\lambda_2}
\ssmove{1} {-2}
\ssdrop{\alpha_1\mu_2}
\ssmove 3 {14}

\ssdrop{\gamma_2(\sigma b)}
\ssmove{1} {-2} 
\ssdrop{\sigma b\lambda_1^{\prime}}

\ssmove{2}{4}
\ssdrop{\alpha_1\gamma_2(\sigma b)}
\ssmove 1 {-2} 
\ssdrop{\alpha_1\lambda_1^{\prime}\sigma b } \ssname{c} 
\ssmove{1}{-6}
\ssdrop{\lambda_2\sigma b}\ssname{d} 
\ssgoto c \ssgoto d \ssstroke
\ssmove 1 {-2} 
\ssdrop{\begin{array}{c}\sigma b \mu_2 \\ \lambda_1^{\prime}\lambda_2\end{array}}
\ssmove{1} {-2}
\ssdrop{\lambda_1^{\prime}\mu_2}
\ssmove 1 6 
\ssdrop{\alpha_1\lambda_2 \sigma b}
\ssmove 1 {-2}
\ssdrop{\begin{array}{c}\alpha_1\sigma b \mu_2 \\ \alpha_1\lambda_1^{\prime}\lambda_2\end{array}}
\ssmove{1}{-2}
\ssdrop{\alpha_1\lambda_1'\mu_2} \ssname{f} 
\ssmove{1}{-6}
\ssdrop{\lambda_2\mu_2} \ssname{e}
\ssgoto e \ssgoto f \ssstroke
\ssmove{2}{-8}
\ssdrop{\mu_2^2}
\end{sseq}
\]
\end{center}
\caption{The $E^{2p^2-5p+3}_{s,t}$-page of the $V(1)$ topological Hochschild-May  spectral sequence for $p\ge 5$ for $s\le 4p^2$ and all $t$. When $p=3$, the only difference is that there is an additional class $\gamma_p(\sigma b)$ in bidegree $(2p^3-2p, 2p^3-3p^2)$.} \label{sseq:2} 
\end{figure}
\begin{proof} 

The spectral sequence \eqref{ss filt coeff} may be expressed as the long exact sequence \eqref{eq les}, or in other words, the diagram 
\[ \xymatrix{  \Sigma^{2p-3} E(\lambda_1', \lambda_2)\otimes P(\mu_2)\otimes  \Gamma(\sigma b)  \ar[r] & E(\alpha_1, \lambda_1\gamma_{p-1}(\sigma \alpha_1), (\mu_1)^{p-1}\sigma \tilde{v}_1 ) \otimes P((\mu_1)^p)\otimes \Gamma(\sigma b)  \ar[d] \\ 
&E(\lambda_1', \lambda_2)\otimes P(\mu_2)\otimes  \Gamma(\sigma b)  \ar@{-->}[ul]  } \]
where the dotted arrow indicates a shift in degree by $1$. In particular, in degree $2p^2-1$ and $2p^2-2$ we have the exact sequence
\[    0 \lra  \mathbb{F}_p\{(\mu_1)^{p-1}\sigma \tilde{v}_1\} \lra   \mathbb{F}_p\{\lambda_2\} \lra \mathbb{F}_p\{ \lambda_1' \}\lra \mathbb{F}_p\{\alpha_1 \lambda_1\gamma_{p-1}(\sigma\alpha_1) \} \lra 0. \]
We can therefore determine if there should be a differential in the $V(1)$ topological Hochschild-May  spectral sequence, as stated, by determining if the map 
$ \mathbb{F}_p\{\lambda_2 \} \lra \mathbb{F}_p\{ \lambda_1' \} $
is nontrivial. 
To determine this, we note that the boundary map is exactly the map 
\[ V(1)_*(THH(j; \ell )) \lra  V(1)_*(THH(j;\Sigma ^{2p-2}\ell )) \]
induced by the map $\ell \lra \Sigma^{2p-2}\ell$ given by $1-\psi_q$  where $q$ is the $q$-th Adams operation. This map induces multiplication by $P^1$ in homology 
\[ \xymatrix{ H\mathbb{F}_p^*(\Sigma^{2p-2}\ell) = \Sigma^{2p-2} \mathcal{A}//E(1) \ar[r]^(.6){P^1} &\mathcal{A}//E(1)=H\mathbb{F}_p^*(\ell) .}\]
In the dual, we therefore know that the map 
\[  \xymatrix{ P(1)\co {H\mathbb{F}_p}_*(\ell) = ( \mathcal{A}//E(1) )_* \ar[r] &\Sigma^{2p-2}(\mathcal{A}//E(1))_*={H\mathbb{F}_p}_*(\Sigma^{2p-2}\ell) }\]
sends classes of the form $\bar{\xi}_1y$ to $y$ and the map sends all other classes to zero. The same will therefore be true for the induced map, which we also denote $P(1)$, in the diagram
\[  
\xymatrix{ 
{H\mathbb{F}_p}_*(V(1)\wedge THH(j;\ell)) \ar[r] & {H\mathbb{F}_p}_*(V(1)\wedge THH(j;\Sigma^{2p-2}\ell))  \\ 
 {H\mathbb{F}_p}_*(V(1)\wedge THH(j) )\otimes_{ {H\mathbb{F}_p}_*(j) } {H\mathbb{F}_p}_*(\ell)  \ar[u]^{\cong}  \ar[r]^(.45){P(1) }&    {H\mathbb{F}_p}_*(V(1)\wedge THH(j) )\otimes_{{H\mathbb{F}_p}_*(j)} {H\mathbb{F}_p}_*(\Sigma^{2p-2} \ell)     \ar[u]^{\cong}  ,    }\] 
in particular $\bar{\xi}_1\sigma \tilde{\xi}_1^p$ maps to $\sigma\tilde{\xi}_1^p$. We therefore examine the square 
\[ 
\xymatrix{
V(1)_{2p^2-1}(THH(j;\ell))\ar[r] \ar[d]^{g} & V(1)_{2p^2-2} (THH(j; \Sigma^{2p-3}\ell)) \ar[d]^{h} \\ 
(H\mathbb{F}_p\wedge V(1))_{2p^2-1}(THH(j;\ell))\ar[r] & (H\mathbb{F}_p\wedge V(1))_{2p^2-2} (THH(j; \Sigma^{2p-3}\ell)),} \]
which is isomorphic to 
\[ 
\xymatrix{ 
\mathbb{F}_p\{\lambda_2\} \ar[d]^{g} \ar[r] & \mathbb{F}_p\{\lambda_1'\} \ar[d]^h \\
\mathbb{F}_p\{\sigma \tilde{\xi}_2, \bar{\xi}_1\sigma \tilde{\xi}_1^p, \bar{\tau}_1\sigma b, \bar{\xi}_1\bar{\tau}_0 \sigma b\} \ar[r] & \mathbb{F}_p\{\sigma\tilde{\xi}_1^p, \bar{\tau}_0\sigma b\} . 
}\]
As stated in the proof of Proposition \ref{prop d1}, the vertical maps are defined by 
\begin{center}
$g( \lambda_2 ) =  \sigma \tilde{\xi}_2- \bar{\xi}_1\sigma \tilde{\xi}_1^p -\bar{\tau}_1\sigma b, $\\
$h( \lambda_1' )= \sigma\tilde{\xi}_1^p-\bar{\tau}_0\sigma b.$ \\
\end{center}
 The bottom horizontal map sends the class in the image of $\sigma \tilde{\xi}_2$ to the class $\sigma \tilde{\xi}_1^p$; i.e., 
 \[ \xymatrix{\sigma \tilde{\xi}_2- \bar{\xi}_1\sigma \tilde{\xi}_1^p -\bar{\tau}_1\sigma b \ar@{|->}[r] & \sigma \tilde{\xi}_1^p. }\] 
 Since the inverse image of the Hurewicz map evaluated on this element is 
 \[ h^{-1}(\sigma \tilde{\xi}_1^p)=h^{-1}(\sigma \tilde{\xi}_1^p-\bar{\tau}_0\sigma b)=\lambda_1'. \]
 This proves that the top horizontal map is nontrivial and therefore, there must be a differential 
\[ d^{2p^2-5p+3}((\mu_1)^{p-1}\sigma \tilde{v}_1)\dot{=} \alpha_1 \lambda_1\gamma_{p-1}(\sigma\alpha_1) \]
as stated.
\end{proof} 
\begin{rem}
Due to Oka \cite[Thm. 4.4]{MR614695}, the obstruction to a ring structure on $V(1)$ at the prime $3$ is a composite of maps including the composite map 
\[\xymatrix{ \beta_1\co \Sigma^{11}S \ar[r] & \Sigma^{11}S/p \ar[r]^{\beta_{(1)}}  & S/p \ar[r] &  \Sigma^1 S, }\] 
however we can compute that the induced map $\Sigma^{11}j\rightarrow \Sigma j$ is null homotopic and hence the obstruction vanishes after smashing with $j$. Thus, $V(1)\wedge j$ and hence $V(1)\wedge THH(j)$ are ring spectra, so the ring structure on $V(1)_*THH(j)$ is also correct at the prime $3$. This type of argument is also used in \cite{MR2183525} in the case of $V(1)\wedge ku$. 
\end{rem}

\begin{thm}\label{main thm}
 Let $p>2$ be a prime number and let $V(1)$ be the cofiber of the periodic self-map $v_1\co \Sigma^{2p-2}S/p \rightarrow S/p.$ Then there is an isomorphism
\[ V(1)_*(THH(j))\cong P(\mu_2)\otimes \Gamma(\sigma b) \otimes \mathbb{F}_p\{\alpha_1, \lambda_1', \lambda_2\alpha_1, \lambda_2\lambda_1', \lambda_2\lambda_1'\alpha_1 \}  \]
where the products between the classes 
\[ \{\alpha_1, \lambda_1', \lambda_2\alpha_1, \lambda_2\lambda_1', \lambda_2\lambda_1'\alpha_1 \} \]
are zero except for 
\[ \alpha_1\cdot \lambda_2\lambda_1'=\lambda_1'\cdot \lambda_2\alpha_1 =\lambda_2\lambda_1'\alpha_1.  \] 
\end{thm} 
\begin{proof} 
We can compute the $E^{2p^2-4p+2}$-page by Proposition \ref{prop d_1} and Proposition \ref{prop dp-1}. 
All the algebra generators are in the range $0\le s\le 4p^2$ depicted Figure \ref{sseq:2} except for $\gamma_{p^k}(\sigma b)$ for $k\ge 1$. The figure clearly shows that there no possible differentials $d^r_{s,t}$ for $r>2p^2-4p+2$ in the range $0\le s\le 4p^2$. 
The only remaining possible differentials are differentials with source $\gamma_{p^k}(\sigma b)$ for $k\ge 1$. There are no possible differentials with source $\gamma_{p^k}(\sigma b)$ by a bidegree argument, which we already proved in the proof of Proposition \ref{prop d1}. Therefore, there are no possible differentials with source $\gamma_{p^{k}}(\sigma b)$ and consequently the spectral sequence collapses at the $E_{2p^2-4p+2}$-page. The classes $\gamma_{p^{k}}(\sigma b)$ map to classes of the same name in the Hurewicz image. 
There are no hidden multiplicative extensions on the elements $(\gamma_{p^{k-1}}(\sigma b))^p$ for $k\ge 1$ because the only other classes in this topological degree are in lower filtration in the spectral sequence. To see this consider a general monomial 
\begin{equation}\label{monomial 2} \mu_2^\ell\gamma_j(\sigma b) \alpha_1^{e_1}(\lambda_1^{\prime})^{e_2}(\lambda_2\alpha_1)^{e_3}(\lambda_2\lambda_1')^{e_4}(\lambda_2\lambda_1'\alpha_1)^{e_5}\end{equation}
which is in topological degree 
\[ 2\ell p^2+ (2p^2-2p)j+(2p-3)e_1+(2p^2-2p+1)e_2+(2p^2+2p-4)^{e_3} + (4p^2-2p)e_4+(4p^2-3)e_5 \]
and May filtration
\[ (2p^2-3p)j+(2p-3)e_1+(2p^2-5p+3)e_2+(4p-3)e_3+(2p^2-3p)e_4+(2p^2-p-3)e_5\]
where $\ell,j\ge 0$, $e_i\in\{0,1\}$ for $1\le i\le 5$ and $\sum_{i=1}^5e_i=1$. Then, in order for this monomial to be in the same topological degree as $(\gamma_{p^{k-1}}(\sigma b))^p$ for some $k\ge 1$ the relation
\begin{equation}\label{rel 1} 2\ell p^2+ (2p^2-2p)j+(2p-3)e_1+(2p^2-2p+1)e_2+(2p^2+2p-4)^{e_3} + (4p^2-2p)e_4+(4p^2-3)e_5= (2p^2-2p)p^k\end{equation}
must hold. This already implies that $e_1=e_2=e_5=0$ since $\sum_{i=1}^5e_i=1$. Thus, the relation reduces to
\[2\ell p^2+ (2p^2-2p)j+(2p^2+2p-4)^{e_3} + (4p^2-2p)e_4-(2p^2-2p)p^k=0.\]
We also know that in order for the May filtration of the monomial \eqref{monomial 2} to be greater than that of $\gamma_{p^k}(\sigma b)$, the inequality
\begin{equation}\label{rel 2} (2p^2-3p)j+(4p-3)e_3+(2p^2-3p)e_4 > (2p^2-3p)p^k \end{equation}
must hold. By rearranging terms we see that 
\[(2p^2-2p)j+(2p^2-2p)e_4 -(2p^2-2p)p^k >-(4p-3)e_3+pj+pe_4+p^{k+1}\]
and therefore 
\[
\begin{array}{c} 2\ell p^2+ (2p^2-2p)j+(2p^2+2p-4)^{e_3} + (4p^2-2p)e_4-(2p^2-2p)p^k>  \\
2\ell p^2+(2p^2+2p-4)^{e_3} + 2p^2e_4-(4p-3)e_3+pj+pe_4+p^{k+1}
\end{array}
\]
where the right hand side of the inequality reduces to 
\[2\ell p^2+(2p^2-2p-1)^{e_3} + 2p^2e_4+pj+pe_4+p^{k+1}>0 \]
since $2p^2>2p+1$ for all primes $p$.
Thus, the relations \eqref{rel 1} and \eqref{rel 2} cannot both hold. 
Therefore, there are no elements in May filtration greater than that of $(\gamma_{p^{k-1}}(\sigma b))^p=0$, which has the same May filtration as $\gamma_{p^k}(\sigma b)$. 
This implies that there cannot be a hidden multiplicative extension on $(\gamma_{p^{k-1}}(\sigma b))^p$ or any of the other products $\gamma_{i}(\sigma b)\gamma_j(\sigma b)=0$ where $i+j=p^k$ for $i,j,k\ge 1$. These are the only other possible hidden multiplicative extensions, so no further hidden multiplicative extensions occur.  
\end{proof} 
\begin{rem}
This paper mainly considers mod $(p,v_1)$ homotopy of topological Hochschild homology of the connective cover of the $K(1)$-local sphere, but one could also consider mod $(p,v_1)$ homotopy of topological Hochschild homology of the non-connective $K(1)$-local sphere. This is trivial however, since 
\[ (S/p)_*(L_{K(1)}S)\cong P(v_1^{\pm 1})\otimes E(\alpha_1) \]
 and therefore the map induced by $v_1$ is multiplication by a unit on homotopy groups and therefore $V(1)_*L_{K(1)}S\cong 0$. Since $THH(L_{K(1)}S)$ is a $L_{K(1)}S$-algebra, the spectrum $V(1)\wedge THH(L_{K(1)}S)$ is contractible. It would still be interesting to study $S/p_*THH(L_{K(1)}S)$, but this would not help us approach $S/p_*K(L_{K(1)}S)$ because Theorem \ref{DGM thm} only holds for connective spectra. 
\end{rem} 

\bibliographystyle{plain}
\bibliography{sources}

\begin{bibdiv}
\begin{biblist}

\bib{MR2499538}{book}{
       title={Guido's book of conjectures},
      series={Monographies de L'Enseignement Math\'ematique [Monographs of
  L'Enseignement Math\'ematique]},
   publisher={L'Enseignement Math\'ematique, Geneva},
        date={2008},
      volume={40},
        ISBN={2-940264-07-4},
        note={A gift to Guido Mislin on the occasion of his retirement from
  ETHZ June 2006, Collected by Indira Chatterji},
      review={\MR{2499538}},
}

\bib{thhmay}{article}{
      author={Angelini-Knoll, Gabe},
      author={Salch, Andrew},
       title={A {M}ay-type spectral sequence for higher topological
  {H}ochschild homology},
        date={2018},
        ISSN={1472-2747},
     journal={Algebr. Geom. Topol.},
      volume={18},
      number={5},
       pages={2593\ndash 2660},
  url={https://doi-org-proxy2-cl-msu-edu.proxy1.cl.msu.edu/10.2140/agt.2018.18.2593},
      review={\MR{3848395}},
}

\bib{MR2171809}{article}{
      author={Angeltveit, Vigleik},
      author={Rognes, John},
       title={Hopf algebra structure on topological {H}ochschild homology},
        date={2005},
        ISSN={1472-2747},
     journal={Algebr. Geom. Topol.},
      volume={5},
       pages={1223\ndash 1290 (electronic)},
         url={http://dx.doi.org/10.2140/agt.2005.5.1223},
      review={\MR{2171809 (2007b:55007)}},
}

\bib{MR2183525}{article}{
      author={Ausoni, Christian},
       title={Topological {H}ochschild homology of connective complex
  {$K$}-theory},
        date={2005},
        ISSN={0002-9327},
     journal={Amer. J. Math.},
      volume={127},
      number={6},
       pages={1261\ndash 1313},
  url={http://muse.jhu.edu/journals/american_journal_of_mathematics/v127/127.6ausoni.pdf},
      review={\MR{2183525 (2006k:55016)}},
}

\bib{MR1947457}{article}{
      author={Ausoni, Christian},
      author={Rognes, John},
       title={Algebraic {$K$}-theory of topological {$K$}-theory},
        date={2002},
        ISSN={0001-5962},
     journal={Acta Math.},
      volume={188},
      number={1},
       pages={1\ndash 39},
         url={http://dx.doi.org/10.1007/BF02392794},
      review={\MR{1947457 (2004f:19007)}},
}

\bib{MR2928844}{article}{
      author={Ausoni, Christian},
      author={Rognes, John},
       title={Algebraic {$K$}-theory of the first {M}orava {$K$}-theory},
        date={2012},
        ISSN={1435-9855},
     journal={J. Eur. Math. Soc. (JEMS)},
      volume={14},
      number={4},
       pages={1041\ndash 1079},
         url={http://dx.doi.org/10.4171/JEMS/326},
      review={\MR{2928844}},
}

\bib{MR2079370}{incollection}{
      author={Baas, Nils~A.},
      author={Dundas, Bj{\o}rn~Ian},
      author={Rognes, John},
       title={Two-vector bundles and forms of elliptic cohomology},
        date={2004},
   booktitle={Topology, geometry and quantum field theory},
      series={London Math. Soc. Lecture Note Ser.},
      volume={308},
   publisher={Cambridge Univ. Press, Cambridge},
       pages={18\ndash 45},
         url={http://dx.doi.org/10.1017/CBO9780511526398.005},
      review={\MR{2079370}},
}

\bib{Bar14}{misc}{
      author={Barwick, Clark},
       title={Red-shift and higher categories},
        date={2014},
  url={https://www.msri.org/workshops/689/schedules/18234/documents/2047/assets/20469},
        note={preprint from talk at MSRI},
}

\bib{MR2413133}{article}{
      author={Blumberg, Andrew~J.},
      author={Mandell, Michael~A.},
       title={The localization sequence for the algebraic {$K$}-theory of
  topological {$K$}-theory},
        date={2008},
        ISSN={0001-5962},
     journal={Acta Math.},
      volume={200},
      number={2},
       pages={155\ndash 179},
         url={http://dx.doi.org/10.1007/s11511-008-0025-4},
      review={\MR{2413133 (2009f:19003)}},
}

\bib{Boa98}{incollection}{
      author={Boardman, J.~Michael},
       title={Conditionally convergent spectral sequences},
        date={1999},
   booktitle={Homotopy invariant algebraic structures ({B}altimore, {MD},
  1998)},
      series={Contemp. Math.},
      volume={239},
   publisher={Amer. Math. Soc., Providence, RI},
       pages={49\ndash 84},
  url={https://doi-org-proxy2-cl-msu-edu.proxy1.cl.msu.edu/10.1090/conm/239/03597},
      review={\MR{1718076}},
}

\bib{bok2}{unpublished}{
      author={B{\"o}kstedt, M.},
       title={The topological {H}ochschild homology},
        date={1987},
        note={preprint},
}

\bib{bok}{unpublished}{
      author={B{\"o}kstedt, M.},
       title={The topological {H}ochschild homology of $\mathbb{Z}$ and of
  $\mathbb{Z}/p\mathbb{Z}$},
        date={1987},
        note={preprint},
}

\bib{Day}{thesis}{
      author={Day, Brian~J.},
       title={Construction of biclosed categories},
        type={Ph.D. Thesis},
        date={1970},
}

\bib{MR2030586}{article}{
      author={Devinatz, Ethan~S.},
      author={Hopkins, Michael~J.},
       title={Homotopy fixed point spectra for closed subgroups of the {M}orava
  stabilizer groups},
        date={2004},
        ISSN={0040-9383},
     journal={Topology},
      volume={43},
      number={1},
       pages={1\ndash 47},
         url={http://dx.doi.org/10.1016/S0040-9383(03)00029-6},
      review={\MR{2030586 (2004i:55012)}},
}

\bib{MR960945}{article}{
      author={Devinatz, Ethan~S.},
      author={Hopkins, Michael~J.},
      author={Smith, Jeffrey~H.},
       title={Nilpotence and stable homotopy theory. {I}},
        date={1988},
        ISSN={0003-486X},
     journal={Ann. of Math. (2)},
      volume={128},
      number={2},
       pages={207\ndash 241},
         url={http://dx.doi.org/10.2307/1971440},
      review={\MR{960945}},
}

\bib{DGM}{book}{
      author={Dundas, Bj{\o}rn~Ian},
      author={Goodwillie, Thomas~G.},
      author={McCarthy, Randy},
       title={The local structure of algebraic {K}-theory},
      series={Algebra and Applications},
   publisher={Springer-Verlag London, Ltd., London},
        date={2013},
      volume={18},
        ISBN={978-1-4471-4392-5; 978-1-4471-4393-2},
      review={\MR{3013261}},
}

\bib{MR1307900}{article}{
      author={Dundas, Bj{\o}rn~Ian},
      author={McCarthy, Randy},
       title={Stable {$K$}-theory and topological {H}ochschild homology},
        date={1994},
        ISSN={0003-486X},
     journal={Ann. of Math. (2)},
      volume={140},
      number={3},
       pages={685\ndash 701},
         url={http://dx.doi.org/10.2307/2118621},
      review={\MR{1307900 (96e:19005a)}},
}

\bib{MR2875843}{article}{
      author={Geisser, Thomas},
      author={Hesselholt, Lars},
       title={On a conjecture of {V}orst},
        date={2012},
        ISSN={0025-5874},
     journal={Math. Z.},
      volume={270},
      number={1-2},
       pages={445\ndash 452},
  url={https://doi-org-proxy2-cl-msu-edu.proxy1.cl.msu.edu/10.1007/s00209-010-0806-2},
      review={\MR{2875843}},
}

\bib{MR1410465}{article}{
      author={Hesselholt, Lars},
      author={Madsen, Ib},
       title={On the {$K$}-theory of finite algebras over {W}itt vectors of
  perfect fields},
        date={1997},
        ISSN={0040-9383},
     journal={Topology},
      volume={36},
      number={1},
       pages={29\ndash 101},
         url={http://dx.doi.org/10.1016/0040-9383(96)00003-1},
      review={\MR{1410465 (97i:19002)}},
}

\bib{MR3328537}{incollection}{
      author={Hopkins, Michael~J.},
       title={{$K(1)$}-local {$E_\infty$}-ring spectra},
        date={2014},
   booktitle={Topological modular forms},
      series={Math. Surveys Monogr.},
      volume={201},
   publisher={Amer. Math. Soc., Providence, RI},
       pages={287\ndash 302},
         url={http://dx.doi.org/10.1090/surv/201/16},
      review={\MR{3328537}},
}

\bib{MR1652975}{article}{
      author={Hopkins, Michael~J.},
      author={Smith, Jeffrey~H.},
       title={Nilpotence and stable homotopy theory. {II}},
        date={1998},
        ISSN={0003-486X},
     journal={Ann. of Math. (2)},
      volume={148},
      number={1},
       pages={1\ndash 49},
         url={http://dx.doi.org/10.2307/120991},
      review={\MR{1652975}},
}

\bib{MR1600246}{book}{
      author={Loday, Jean-Louis},
       title={Cyclic homology},
     edition={Second},
      series={Grundlehren der Mathematischen Wissenschaften [Fundamental
  Principles of Mathematical Sciences]},
   publisher={Springer-Verlag, Berlin},
        date={1998},
      volume={301},
        ISBN={3-540-63074-0},
         url={http://dx.doi.org/10.1007/978-3-662-11389-9},
        note={Appendix E by Mar{\'{\i}}a O. Ronco, Chapter 13 by the author in
  collaboration with Teimuraz Pirashvili},
      review={\MR{1600246 (98h:16014)}},
}

\bib{MR1473888}{article}{
      author={McClure, J.},
      author={Schw{\"a}nzl, R.},
      author={Vogt, R.},
       title={{$THH(R)\cong R\otimes S^1$} for {$E_\infty$} ring spectra},
        date={1997},
        ISSN={0022-4049},
     journal={J. Pure Appl. Algebra},
      volume={121},
      number={2},
       pages={137\ndash 159},
         url={http://dx.doi.org/10.1016/S0022-4049(97)00118-7},
      review={\MR{1473888}},
}

\bib{MR1164148}{article}{
      author={McClure, J.~E.},
      author={Staffeldt, R.~E.},
       title={The chromatic convergence theorem and a tower in algebraic
  {$K$}-theory},
        date={1993},
        ISSN={0002-9939},
     journal={Proc. Amer. Math. Soc.},
      volume={118},
      number={3},
       pages={1005\ndash 1012},
         url={http://dx.doi.org/10.2307/2160154},
      review={\MR{1164148 (93j:55012)}},
}

\bib{MR1209233}{article}{
      author={McClure, J.~E.},
      author={Staffeldt, R.~E.},
       title={On the topological {H}ochschild homology of {$b{\rm u}$}. {I}},
        date={1993},
        ISSN={0002-9327},
     journal={Amer. J. Math.},
      volume={115},
      number={1},
       pages={1\ndash 45},
         url={http://dx.doi.org/10.2307/2374721},
      review={\MR{1209233 (94d:55020)}},
}

\bib{MR614695}{article}{
      author={Oka, Shichir{\^o}},
       title={Ring spectra with few cells},
        date={1979},
     journal={Japan. J. Math. (N.S.)},
      volume={5},
      number={1},
       pages={81\ndash 100},
      review={\MR{614695}},
}

\bib{MR0315016}{article}{
      author={Quillen, Daniel},
       title={On the cohomology and {$K$}-theory of the general linear groups
  over a finite field},
        date={1972},
        ISSN={0003-486X},
     journal={Ann. of Math. (2)},
      volume={96},
       pages={552\ndash 586},
  url={https://doi-org-proxy2-cl-msu-edu.proxy1.cl.msu.edu/10.2307/1970825},
      review={\MR{0315016}},
}

\bib{rav2}{book}{
      author={Ravenel, Douglas~C.},
       title={Nilpotence and periodicity in stable homotopy theory},
      series={Annals of Mathematics Studies},
   publisher={Princeton University Press, Princeton, NJ},
        date={1992},
      volume={128},
        ISBN={0-691-02572-X},
        note={Appendix C by Jeff Smith},
      review={\MR{1192553 (94b:55015)}},
}

\bib{150102547S}{article}{
      author={{Salch}, A.},
       title={{The Hochschild homology of A(1)}},
        date={2015-01},
     journal={ArXiv e-prints},
      eprint={1501.02547},
}

\bib{schwedebook}{article}{
      author={Schwede, Stefan},
       title={Symmetric spectra book project},
        date={available online},
     journal={draft version},
         url={http://www.math.uni-bonn.de/people/schwede/SymSpec.pdf},
}

\bib{MR764579}{incollection}{
      author={Waldhausen, Friedhelm},
       title={Algebraic {$K$}-theory of spaces, localization, and the chromatic
  filtration of stable homotopy},
        date={1984},
   booktitle={Algebraic topology, {A}arhus 1982 ({A}arhus, 1982)},
      series={Lecture Notes in Math.},
      volume={1051},
   publisher={Springer, Berlin},
       pages={173\ndash 195},
         url={http://dx.doi.org/10.1007/BFb0075567},
      review={\MR{764579 (86c:57016)}},
}

\bib{MR802796}{incollection}{
      author={Waldhausen, Friedhelm},
       title={Algebraic {$K$}-theory of spaces},
        date={1985},
   booktitle={Algebraic and geometric topology ({N}ew {B}runswick, {N}.{J}.,
  1983)},
      series={Lecture Notes in Math.},
      volume={1126},
   publisher={Springer},
     address={Berlin},
       pages={318\ndash 419},
         url={http://dx.doi.org/10.1007/BFb0074449},
      review={\MR{802796 (86m:18011)}},
}

\bib{MR1269324}{book}{
      author={Weibel, Charles~A.},
       title={An introduction to homological algebra},
      series={Cambridge Studies in Advanced Mathematics},
   publisher={Cambridge University Press, Cambridge},
        date={1994},
      volume={38},
        ISBN={0-521-43500-5; 0-521-55987-1},
         url={http://dx.doi.org/10.1017/CBO9781139644136},
      review={\MR{1269324 (95f:18001)}},
}

\end{biblist}
\end{bibdiv}
\end{document}